\documentclass[11pt,reqno]{amsart}
\usepackage{geometry} 
\usepackage{amsmath,amssymb,amsthm, mathrsfs}
\usepackage{xcolor}
\usepackage{graphicx} 

\usepackage{enumerate}
\usepackage{indentfirst}
\usepackage[colorlinks,
            linkcolor=black,
            anchorcolor=black,
            citecolor=black
            ]{hyperref}
\DeclareMathOperator{\diver}{\mathrm{div}}

\newtheorem{mydef}{Definition}
\newtheorem{thm}{Theorem}[section]
\newtheorem{lem}{Lemma}[section]

\newtheorem{rk}{Remark}
\newtheorem{cor}{Corollary}[section]
\numberwithin{rk}{section}
\numberwithin{prop}{section}
\numberwithin{mydef}{section}
\numberwithin{lem}{section}
\numberwithin{equation}{section}
\numberwithin{thm}{section}
\allowdisplaybreaks[2]

\title[degenerate compressible  Navier-Stokes equations]{Global  regular solutions for  1-D degenerate compressible  Navier-Stokes equations with large data and  far field vacuum}
\date{\today}

\author{Yue Cao}
\address[Yue Cao]{School of Mathematical Sciences, and MOE-LSC, Shanghai Jiao Tong University,
Shanghai 200240, P. R. China} \email{\tt cao\_yuecdut@163.com}

\author{Hao Li}
\address[Hao Li]{School of Mathematical Sciences,  Fudan University, Shanghai 200433, P. R. China} \email{\tt  hao\_li@fudan.edu.cn}

\author{Shengguo Zhu }
\address[Shengguo Zhu]{School of Mathematical Sciences, CMA-Shanghai,  and MOE-LSC,  Shanghai Jiao Tong University, Shanghai 200240, P. R. China}
 \email{\tt  zhushengguo@sjtu.edu.cn}

\begin{document}
\begin{abstract}

In this paper, the Cauchy problem for the  one-dimensional (1-D) isentropic  compressible  Navier-Stokes equations (\textbf{CNS}) is considered. When the viscosity $\mu(\rho)$  depends on the 
density $\rho$ in a sublinear power law ($ \rho^\delta$ with $0<\delta\leq 1$), based on an elaborate analysis  of  the intrinsic  singular  structure of this degenerate system,   we prove the  global-in-time well-posedness of  regular solutions with  conserved  total  mass, momentum, and  finite  total  energy in some inhomogeneous Sobolev spaces.   Moreover, the solutions we obtained  satisfy   that   $\rho$  keeps positive for all point $x\in \mathbb{R}$  but decays to zero in the far field, which is consistent with the facts  that the total  mass of  the whole space is conserved, and  \textbf{CNS} is a model of non-dilute fluids where $\rho$ is bounded below away from zero.   The key to the proof is the introduction of a well-designed reformulated structure   by introducing some new variables and initial compatibility conditions, which, actually,  can transfer the degeneracies of the time evolution and the viscosity  to the possible singularity of some special source terms. Then, combined  with the  BD entropy estimates and    transport properties of the  so-called effective velocity $v=u+\varphi(\rho)_x$ ($u$ is the velocity of the fluid, and $\varphi(\rho)$ is a function of $\rho$ defined by $\varphi'(\rho)=\mu(\rho)/\rho^2$), one can obtain the required uniform a priori estimates of  corresponding solutions.   Moreover,   in contrast to the classical  theory in the case of  the constant viscosity, one can  show   that  the  $L^\infty$ norm of $u$ of the global regular solution we obtained does not  decay to zero as time $t$ goes to infinity. It is worth pointing out that  the well-posedness theory established here  can be applied to  the viscous Saint-Venant system for the motion of  shallow water.

\end{abstract}
\subjclass[2010]{35A01, 35Q30, 35B65, 35A09.}
\keywords{
Degenerate compressible Navier-Stokes equations, One-dimension, Regular solution, Far field vacuum, Large data, Global-in-time well-posedness.}
\maketitle

\section{introduction}

The time evolution of the mass density $\rho\geq 0$ and the velocity $u=\left(u^{(1)},u^{(2)},\cdots, u^{(d)}\right)^\top$ $\in \mathbb{R}^d$ of a general viscous isentropic compressible fluid occupying a spatial domain $\Omega\subset \mathbb{R}^d$ is governed by the following isentropic compressible Navier-Stokes equations (\textbf{ICNS}):
\begin{equation}
\label{eq:1.1}
\begin{cases}
\rho_t+\text{div}(\rho u)=0,\\[4pt]
(\rho u)_t+\text{div}(\rho u\otimes u)
+\nabla
P =\text{div} \mathbb{T}.
\end{cases}
\end{equation}
Here,  $x=(x_1,x_2,\cdots, x_d)^{\top}\in \Omega$, $t\geq 0$ are the space and time variables, respectively.  
 For the polytropic gases, the constitutive relation is given by
\begin{equation}
\label{eq:1.2}
P=A\rho^{\gamma}, \quad A>0,\quad  \gamma> 1,
\end{equation}
where $A$ is  an entropy  constant and  $\gamma$ is the adiabatic exponent. $\mathbb{T}$ denotes the viscous stress tensor with the  form
\begin{equation}
\label{eq:1.3}
\begin{split}
&\mathbb{T}=2\mu(\rho)D(u)+\lambda(\rho)\text{div}u\,\mathbb{I}_d,
\end{split}
\end{equation}
 where $D(u)=\frac{1}{2}\big(\nabla u+(\nabla u)^\top\big)$ is the deformation tensor,    $\mathbb{I}_d$ is the $d\times d$ identity matrix,
\begin{equation}
\label{fandan}
\mu(\rho)=\alpha  \rho^\delta,\quad \lambda(\rho)=\beta  \rho^\delta,
\end{equation}
for some  constant $\delta\geq 0$,
 $\mu(\rho)$ is the shear viscosity coefficient, $\lambda(\rho)+\frac{2}{d}\mu(\rho)$ is the bulk viscosity coefficient,  $\alpha$ and $\beta$ are both constants satisfying
 \begin{equation}\label{10000}\alpha>0 \quad \text{and} \quad   2\alpha+d\beta\geq 0.
 \end{equation}

In the theory of gas dynamics, the \textbf{CNS} can be derived from the Boltzmann equations through the Chapman-Enskog expansion, cf. Chapman-Cowling \cite{chap} and Li-Qin \cite{tlt}. Under some proper physical assumptions, one can find that  the viscosity coefficients and heat conductivity coefficient $\kappa$ are not constants but functions of the absolute temperature $\theta$ such as:
\begin{equation}
\label{eq:1.5g}
\begin{split}
\mu(\theta)=&a_1 \theta^{\frac{1}{2}}F(\theta),\quad \lambda(\theta)=a_2 \theta^{\frac{1}{2}}F(\theta), \quad \kappa(\theta)=a_3 \theta^{\frac{1}{2}}F(\theta),
\end{split}
\end{equation}
for some constants $a_i$ $(i=1,2,3)$ (see \cite{chap}). Actually for the cut-off inverse power force models, if the intermolecular potential varies as $r^{-\nu}$,
where $ r$ is intermolecular distance, then in (\ref{eq:1.5g}):
$$F(\theta)=\theta^{b}\quad \text{with}\quad b=\frac{2}{\nu} \in [0,\infty).$$
In particular (see \S 10 of \cite{chap}), for ionized gas,
$\nu=1$ and $ b=2$;
for Maxwellian molecules,
$\nu=4$ and $ b=\frac{1}{2}$; 
while for rigid elastic spherical molecules,
$\nu=\infty$ and $  b=0$.

According to Liu-Xin-Yang \cite{taiping}, for isentropic and polytropic fluids, such a dependence is inherited through the laws of Boyle and Gay-Lussac:
$$
P=R\rho \theta=A\rho^\gamma \quad \text{for \  constant} \quad R>0,
$$i.e., $\theta=AR^{-1}\rho^{\gamma-1}$, and one can see that the viscosity coefficients are functions of $\rho$ of  the form $(\ref{fandan})$.   Actually, there do exist some physical models that satisfy the density-dependent viscosities  assumption \eqref{fandan}, such as Korteweg system, Shallow water equations, lake equations and quantum Navier-Stokes system (see \cite{bd6, bdvv,bd2, bd,  BN2,  Mar,vy2}).\smallskip

In the current paper,   the following   1-D degenerate    \textbf{ICNS} is considered:
\begin{equation}\label{e1.1}
\left\{
\begin{aligned}
&\rho_t+(\rho u)_x=0,\\[4pt]
&(\rho u)_t+(\rho u^2)_x+ P(\rho)_x=(\mu(\rho)u_x)_x,
\end{aligned}
\right.
\end{equation}
where the  pressure $P$ and the viscosity coefficient $\mu(\rho)$ are given by \eqref{eq:1.2} and \eqref{fandan}, respectively.
 We are concerned with the  global well-posedness of   regular solutions with large data and far field vacuum  to the Cauchy problem of  \eqref{e1.1} with the following initial data and far field behavior:
\begin{equation}\label{e1.3}
\begin{split}
\displaystyle
(\rho(0,x),u(0,x))=(\rho_0(x),u_0(x)) \quad  & \text{for}\quad  x\in\mathbb R,\\[4pt]
\displaystyle
\left(\rho(t,x),u(t,x)\right)\to \left(0,0\right)\quad  \text{as}\quad  \left|x\right|\to \infty\quad & \text{for}\quad   t\ge 0.
\end{split}
\end{equation}


 Throughout this paper, we adopt the following simplified notations, most of them are for the standard homogeneous and inhomogeneous Sobolev spaces:
\begin{equation*}\begin{split}
\displaystyle
 & \|f\|_s=\|f\|_{H^s(\mathbb{R}^d)},\quad |f|_p=\|f\|_{L^p(\mathbb{R}^d)},\quad \|f\|_{m,p}=\|f\|_{W^{m,p}(\mathbb{R}^d)},\\[6pt]
 \displaystyle
 & D^{k,r}=\{f\in L^1_{loc}(\mathbb{R}^d): |f|_{D^{k,r}}=|\nabla^kf|_{r}<\infty\},\quad D^k=D^{k,2},  \\[6pt]
 \displaystyle
  & D^{1}_*=\{f\in L^6(\mathbb{R}^3):  |f|_{D^1_*}= |\nabla f|_{2}<\infty\},\\[6pt]
  \displaystyle
&  |f|_{D^1_*}=\|f\|_{D^1_*(\mathbb{R}^3)},\quad \|f\|_{L^p([0,T];L^q)} = \|f\|_{L^p([0,T]; L^q(\mathbb{R}^d))},\\[6pt]
  \displaystyle
 & \|f\|_{X_1 \cap X_2}=\|f\|_{X_1}+\|f\|_{X_2}, \quad    \int f\text{d}x=\int_{\mathbb{R}^d}  f \text{d}x, \\[6pt]
 &  X([0,T]; Y)=X([0,T]; Y(\mathbb{R}^d)),\quad \|(f,g)\|_X=\|f\|_{X}+\|g\|_{X}.\end{split}
\end{equation*}
We will clearly indicate the value of $d$ where the above notations are used.
 A detailed study of homogeneous Sobolev spaces  can be found in \cite{gandi}.

 For the  flows of constant viscosities  ($\delta=0$ in (\ref{fandan})), there is a lot of literature on the  well-posedness  of strong/classical solutions to the system   \eqref{eq:1.1}.    When  $\inf_x {\rho_0(x)}>0$, the local well-posedness of three-dimensonal (3-D) classical solutions to the Cauchy problem of  \eqref{eq:1.1} follows from the standard symmetric hyperbolic-parabolic structure satisfying  the well-known Kawashima's condition, cf. \cite{KA, nash, serrin}, which has been extended to be a global one by Matsumura-Nishida \cite{mat} for initial data close to a non-vacuum equilibrium in some Sobolev space $H^s(\mathbb{R}^3)$ ($s>\frac{5}{2}$).  In $\mathbb{R}$, the global well-posedness of strong solutions with arbitrarily large data in some bounded domains has been proven by   Kazhikhov-Shelukhin \cite{KS}, and later, Kawashima-Nishida \cite{KN}  extended this theory  to the unbounded domains. However, these approaches do not work when $\inf_x {\rho_0(x)}=0$, which occurs when some physical requirements are imposed, such as finite total initial mass  and energy in the whole space.  One of the  main issues  in the presence of vacuum is the degeneracy of the time evolution operator, which makes it hard to  understand the behavior of the  velocity field near the vacuum.
A  remedy was suggested  by Cho-Choe-Kim \cite{CK3}, where they imposed  initially a {\it compatibility condition}:
\begin{equation*}
-\text{div} \mathbb{T}_{0}+\nabla P(\rho_0)=\sqrt{\rho_{0}} g \quad \text{for\  \ some} \ \ g\in L^2(\mathbb{R}^3),
\end{equation*}
which  leads to  $\sqrt{\rho}u_t\in L^\infty ([0,T_{*}]; L^2)$  for a  short  time $T_{*}>0$.  Then  they established  successfully   the local well-posedness of  smooth solutions with vacuum  in $\mathbb{R}^3$ (see also Cho-Kim \cite{guahu}), which,  recently,  has been shown to be   a global one with small energy by Huang-Li-Xin \cite{HX1} in  $\mathbb{R}^3$. Later,  Jiu-Li-Ye \cite{j2}  proved that the 1-D Cauchy problem of  \eqref{eq:1.1} admits  a unique global classical solution  with arbitrarily large data and vacuum.

 For density-dependent viscosities ($\delta>0$ in (\ref{fandan})) which degenerate at the  vacuum,  system \eqref{eq:1.1} has received extensive attentions in recent years. When  $\inf_x {\rho_0(x)}>0$, some important progresses  on  the global well-posedness of  classical solutions to the Cauchy problem of  \eqref{eq:1.1}  have been obtained, which include 
  Burtea-Haspot \cite{bh}, Constantin-Drivas-Nguyen-Pasqualotto \cite{cons}, Haspot \cite{HB}, Kang-Vasseur \cite{kv},
   Mellet-Vasseur \cite{vassu2}  for 1-D flow with arbitrarily large data, and Sundbye \cite{Sundbye2}  and Wang-Xu \cite{weike} for two-dimensional (2-D) flow with  initial data close to a non-vacuum equilibrium in some Sobolev spaces.  When vacuum appears, instead of the uniform elliptic
structure, the viscosity degenerates when density vanishes, which raises the difficulty of the
problem to another level.  Actually,  in  \cite{CK3, HX1,j2}, the uniform ellipticity of the Lam\'e operator $L$ defined by
\begin{equation*}Lu =-\alpha\triangle u-(\alpha+\beta)\nabla \mathtt{div}u
\end{equation*}
($Lu=-\alpha u_{xx}$ in $\mathbb{R}$) plays an essential role in the high order regularity estimates on $u$. One can use the standard elliptic theory to estimate $|u|_{D^{k+2}}$  by the $D^k$-norm of all other terms in momentum equations.  However, for $\delta>0$,  viscosity coefficients vanish as the  density function connects to vacuum continuously.  This degeneracy makes it difficult to adapt the approach of  the constant viscosity case in \cite{CK3,HX1,j2} to the current case.
A remarkable discovery of a new mathematical entropy
function was made by Bresch-Desjardins \cite{bd} for
 $\lambda(\rho)$ and $\mu(\rho)$ satisfying the relation
\begin{equation}\label{bds}\lambda(\rho)=2(\mu'(\rho)\rho-\mu(\rho)),
\end{equation}
which offers an estimate
$
\mu'(\rho)\nabla \sqrt{\rho}\in L^\infty ([0,T];L^2(\mathbb{R}^d))
$
provided that 
$ \mu'(\rho_0)\nabla \sqrt{\rho_0}\in L^2(\mathbb{R}^d)$ for any $d\geq 1$.
This observation plays an important role in the development of the global existence of weak solutions with vacuum for system (\ref{eq:1.1}) and some related models, see Bresch-Desjardins \cite{bd6}, Bresch-Vasseur-Yu \cite{bvy}, Jiu-Xin \cite{j1}, Li-Xin \cite{lz}, Mellet-Vasseur \cite{vassu1}, Vasseur-Yu \cite{vayu} and so on. However, the regularities and uniquness of such weak solutions with  vacuum  remain open.

Notice that, for  the cases  $\delta\in (0,\infty)$,  if  $\rho>0$, $\eqref{eq:1.1}_2$ can be formally rewritten as
\begin{equation}\label{qiyi}
\begin{split}
u_t+u\cdot\nabla u +\frac{A\gamma}{\gamma-1}\nabla\rho^{\gamma-1}+\rho^{\delta-1} Lu=\psi \cdot  Q(u),
\end{split}
\end{equation}
where the quantities $\psi$ and $Q(u)$ are given by
\begin{equation}\label{operatordefinition}
\begin{split}
\psi \triangleq & \nabla \ln \rho \quad \text{when}\quad \delta=1;\\
\psi \triangleq & \frac{\delta}{\delta-1} \nabla \rho^{\delta-1}\quad \text{when}\quad \delta \in (0,1)\cup (1,\infty);\\
Q(u)\triangleq &\alpha(\nabla u+(\nabla u)^\top)+\beta\mathtt{div}u\mathbb{I}_d.
\end{split}
\end{equation}
When $\delta=1$, from  (\ref{qiyi})-(\ref{operatordefinition}),  the degeneracy of time evolution and elliptic structure in momentum equations caused by the  vacuum has been transferred to the possible singularity of the  term  $\nabla \ln \rho$, which actually  can be controlled by a symmetric hyperbolic system with a source term  $\nabla \text{div}u$ in  Li-Pan-Zhu \cite{sz3}. Then via establishing the  uniform a priori estimates in $L^6\cap D^1\cap D^2$  for $\nabla \ln \rho$, the existence of 2-D local classical solution with far field vacuum   to  \eqref{eq:1.1} has been obtained in \cite{sz3}, which also applies to the 2-D Shallow water equations.
When $\delta> 1$,  (\ref{qiyi})-(\ref{operatordefinition}) imply that actually the velocity $u$ can be governed by a nonlinear degenerate  parabolic system without singularity   near the vacuum region.
Based on this observation, by using some hyperbolic  approach which bridges the parabolic system (\ref{qiyi}) when $\rho>0$ and the hyperbolic one
$$u_t+u\cdot\nabla u=0 \quad \text{when} \quad  \rho=0,$$   the existence of 3-D local classical solutions with vacuum to  \eqref{eq:1.1} was established in  Li-Pan-Zhu \cite{sz333}.
The corresponding global well-posedness in some homogeneous Sobolev spaces  has been established  by Xin-Zhu \cite{zz} under some initial smallness  assumptions. Recently, for the case $0<\delta<1$, via introducing an elaborate elliptic approach on the operators $L(\rho^{\delta-1}u)$ and some initial compatibility conditions, the existence of 3-D local regular solution with far field vacuum has been obtained by Xin-Zhu  \cite{zz2}. Some other interesting results and discussions can also be seen in Chen-Chen-Zhu \cite{Geng2}, Ding-Zhu \cite{ding}, Geng-Li-Zhu \cite{zhu}, Germain-Lefloch \cite{Germain}, Guo-Li-Xin \cite{guo}, Lions \cite{lions},  Luo-Xin-Zeng \cite{luotao}, Yang-Zhao \cite{zyj}, Yang-Zhu \cite{tyc2}, Zhu \cite{zhuthesis} and so on.  

However,  due to the complicated mathematical structure of the degenerate \textbf{ICNS} and lack of smooth effect on solutions when vacuum appears,  even many important progresses have been obtained, a lot of fundamental questions still remain open, such as  which kind of initial data can cause the finite time blow-up of smooth solutions, or provide  the global existence of smooth solutions  with  conserved  total  mass, momentum, and  finite  total  energy. In the current  paper,  we identify a class of 1-D initial data, which  admits a global  classical  solution with far field vacuum to the Cauchy problem  for the cases $0<\delta\leq 1$  in some inhomogeneous Sobolev spaces.

\smallskip

\subsection{Global-in-time well-posedness of regular solutions}


In order to present our results clearly, we first introduce the following two definitions for the case $0<\delta<1$ and the case $\delta=1$
of regular solutions with far field vacuum  to the Cauchy problem \eqref{e1.1}-\eqref{e1.3}, respectively.

\begin{mydef}\label{def2}
Assume $0<\delta<1$ and  $T> 0$. The pair $(\rho,u)$ is called a regular solution to the Cauchy problem \eqref{e1.1}-\eqref{e1.3} in $[0,T]\times \mathbb R$, if $(\rho, u)$ satisfies this problem in the sense of distributions and:
\begin{equation*}
\begin{aligned}
&(1)\ \inf_{(t,x)\in [0,T]\times \mathbb R}\rho(t,x)=0,\quad 0<\rho^{\gamma-1}\in C([0,T];H^2),\\
& \quad \ \ (\rho^{\delta-1})_x\in C([0,T];H^1);\notag\\
&(2)\ u\in C([0,T];H^2)\cap L^2([0,T]; H^3),\quad  \rho^{\frac{\delta-1}{2}}u_x\in C([0,T]; L^2), \\ 
& \quad \ \ u_t\in C([0,T];L^2)\cap L^2([0,T];H^1),\quad \rho^{\frac{\delta-1}{2}}u_{tx}\in L^2([0,T];L^2),\\
&\quad \ \  \rho^{\delta-1} u_{xx}\in C([0,T];L^2)\cap L^2([0,T];H^1).
\end{aligned}
\end{equation*}
\end{mydef}

\begin{mydef}\label{def1}
Assume $\delta=1$ and $T> 0$. The pair $(\rho,u)$ is called a regular solution to the Cauchy problem \eqref{e1.1}-\eqref{e1.3} in $[0,T]\times \mathbb R$, if $(\rho, u)$ satisfies this problem in the sense of distributions and:
\begin{equation*}
\begin{aligned}
&(1)\ \inf_{(t,x)\in [0,T]\times \mathbb R}\rho(t,x)=0,\quad 0<\rho^{\gamma-1}\in C([0,T];H^2),\\
& \quad \ \ (\ln\rho)_x\in C([0,T];H^1); \notag\\
\notag
&(2)\ u\in C([0,T];H^2)\cap L^2([0,T]; H^3),\quad u_t\in C([0,T];L^2)\cap L^2([0,T];H^1).\notag\\
\end{aligned}
\end{equation*}
\end{mydef}

\begin{rk}\label{r1}
First, it follows from  the Definitions \ref{def2}-\ref{def1} that $
(\rho^{\delta-1})_x\in L^\infty$ for the case $ 0<\delta<1$, and 
 $(\ln \rho)_x\in L^\infty$ for the case $\delta=1$,
 which, along with $\rho^{\gamma-1}\in C([0,T];H^2)$, means that the vacuum occurs if and only if in the far field.

Second, we introduce some physical quantities that will be used in this paper:
\begin{equation*}
\begin{split}
m(t)=&\int \rho(t,x){\rm{d}}x\quad \textrm{(total mass)},\\
\mathbb{P}(t)=&\int \rho(t,x)u(t,x){\rm{d}}x \quad \textrm{(momentum)},\\
E_k(t)=&\frac{1}{2}\int \rho(t,x)u^2(t,x){\rm{d}}x\quad \textrm{ (total kinetic energy)},\\
E(t)=&E_k(t)+\int \frac{P(t,x)}{\gamma-1}{\rm{d}}x \quad \textrm{ (total energy)}.
\end{split}
\end{equation*}
Actually, it follows from the  definitions that the  regular solution satisfies the conservation of  total mass (see Lemma \ref{ms}) and   momentum (see Lemma \ref{lemmak}).  Furthermore, it satisfies the energy equality (see Lemma \ref{l4.1}  and its proof in Appendix B).  Note  that the conservation of  momentum is not  clear   for the  strong solution with vacuum  to the flows of constant viscosities Jiu-Li-Ye \cite{j2} (see an example on non-conservation of momentum in \cite{zz2} ). In this sense, the definitions of regular solutions here are consistent with the physical background of \textbf{CNS}.
\end{rk}

The regular solutions select  velocity in a physically reasonable way when the density approaches the  vacuum at far fields. Under the help of this notion of solutions, the  momentum equations can be reformulated into a special quasi-linear parabolic system with some possible singular source terms near the  vacuum, and  the coefficients in front  of  $u_{xx}$ will tend to $\infty$ as $\rho\rightarrow 0$ in the far field for $0<\delta<1$. However, the problems  become trackable through some   elaborate  arguments on the  energy estimates.

For the case $0<\delta<1$, the global-in-time well-posedness result reads as:

\begin{thm}\label{th2}
Assume $\delta$ and $\gamma$ satisfy 
\begin{equation}\label{canshe}
0<\delta<1,\quad \gamma>1, \quad  \gamma\geq \delta+\frac12.
\end{equation}
If the initial data $(\rho_0, u_0)$ satisfy 
\begin{equation}\label{e1.7}
 0<\rho_0\in L^1, \quad (\rho_0^{\gamma-1}, u_0)\in H^2, \quad (\rho_0^{\delta-1})_x\in H^1,\end{equation}
and the initial compatibility conditions
\begin{equation}\label{e1.8}
\partial_x u_0=\rho_0^{\frac{1-\delta}{2}}g_1,\quad   \alpha \partial_{x}^2 u_0=\rho_0^{1-\delta} g_2,
\end{equation}
for some $g_i\in L^2(i=1,2)$,
then the Cauchy problem \eqref{e1.1}-\eqref{e1.3}  admits a unique global classical solution
$(\rho, u)$  in $(0,\infty)\times \mathbb R$  satisfying that, for any $0<T<\infty$, $(\rho, u)$  is a regular one in  $[0,T]\times \mathbb R$ as defined in Definition \ref{def2}, and  
\begin{equation}\label{er1}
\begin{split}
&\rho\in C([0,T];L^1),\quad (\rho^{\gamma-1})_t\in C([0,T];H^1),\quad (\rho^{\delta-1})_{tx}\in C([0,T];L^2),  \\
&t^{\frac12}\rho^{\frac{\delta-1}{2}}u_{tx}\in L^\infty([0,T];L^2)\cap L^2([0,T];D^1),\quad  t^{\frac12}u_{tt}\in L^2([0,T];L^2).
\end{split}
\end{equation}

\end{thm}

\begin{rk}\label{rk1}
First, observe that there are no smallness conditions imposed on $(\rho_0, u_0)$.

Second, one   can find the following class of initial data $(\rho_0, u_0)$ satisfying the conditions \eqref{e1.7}-\eqref{e1.8}:
$$\rho_0(x)=\frac{1}{1+x^{2\sigma}},\quad u_0(x)\in C_0^2(\mathbb R),$$
where  $  \max\big\{\frac34, \frac{1}{4(\gamma-1)}\big\}<\sigma<\frac{1}{4(1-\delta)}.$
Particularly, when  $\partial_x u_0$ is  compactly supported,
  the compatibility conditions \eqref{e1.8} are satisfied automatically.

Moreover,   the compatibility conditions \eqref{e1.8} are also necessary for the existence of the unique  regular  solution $(\rho, u)$ obtained in Theorem \ref{th2}.
In particular, the one shown in $\eqref{e1.8}_2$  plays a key role in the derivation of $u_t \in L^\infty([0,T_*];L^2(\mathbb{R}))$, which will be used in the uniform estimates for  $|u|_{D^2}$.

\end{rk}



For the case $\delta=1$, the global-in-time well-posedness result reads as:

\begin{thm}\label{th1}
Assume $\delta=1$ and  $\gamma>\frac32$.
If the initial data $(\rho_0, u_0)$ satisfy 
 \begin{equation}\label{id1}
0<\rho_0\in L^1,  \quad (\rho_0^{\gamma-1}, u_0)\in H^2, \quad (\ln\rho_0)_x\in H^1,
\end{equation}
then  the Cauchy problem \eqref{e1.1}-\eqref{e1.3}  admits a unique global classical solution
$(\rho, u)$  in $(0,\infty)\times \mathbb R$  satisfying that, for any $0<T<\infty$, $(\rho, u)$  is a regular one in  $[0,T]\times \mathbb R$ as defined in Definition \ref{def1}, and 
\begin{equation}\label{er2}
\begin{split}
&\rho\in C([0,T];L^1),\quad (\rho^{\gamma-1})_t\in C([0,T];H^1),\quad   (\ln\rho)_{tx}\in C([0,T];L^2),\\
&t^{\frac12}u_{tx}\in L^\infty([0,T];L^2)\cap L^2([0,T];D^1),\quad t^{\frac12}u_{tt}\in L^2([0,T];L^2).
\end{split}
\end{equation}

\end{thm}

\begin{rk}First, observe that there are no smallness conditions imposed on $(\rho_0, u_0)$.

Second,   one  can find the following class of initial data $(\rho_0, u_0)$ satisfying  \eqref{id1}:
$$\rho_0(x)=\frac{1}{1+x^{2\sigma}},\quad u_0(x)\in C_0^2(\mathbb R), $$
where  $ \sigma>\max\big\{\frac34, \frac{1}{4(\gamma-1)}\big\}.$

Moreover,   it  should be pointed out  that the theory established in  Theorem \ref{th1} can be applied to  the viscous \textbf{Saint-Venant} system for the motion of  shallow water (i.e., $\delta=1$, $\gamma=2$, see Gerbeau-Perthame \cite{gpm}).

\end{rk}

\begin{rk}
It is worth pointing out that in the  domain $\Omega=[0,1]$,  Li-Li-Xin \cite{hailiang} proved that the entropy weak solutions for general large initial data satisfying finite  entropy to the corresponding initial-boundary value problem of the system \eqref{e1.1} with $\delta>\frac{1}{2}$  exist globally in time. Moreover, the authors  showed that  for more regular initial data, there
is a global entropy weak solution which is unique and regular with well-defined velocity
field for short time, and the interface of initial vacuum propagates along the particle path
during this time period. Furthermore, the dynamics of weak solutions and vacuum states are investigated rigorously. It is shown that for any global entropy weak solution, any
(possibly existing) vacuum state must vanish within finite time, which means that there exist some time $T_0>0$ (depending on initial data) and a constant $\rho_{-}$ such that 
\begin{equation}\label{vanishvacuum}
\begin{split}
\inf_{x\in \overline{\Omega}}\rho(t,x)\geq \rho_{-}>0 \quad \text{for}\quad t\geq T_0.
\end{split}
\end{equation}
The velocity (even if
regular enough and well-defined) blows up in finite time as the vacuum states vanish, and  after the vanishing of vacuum states, the global entropy weak solution
becomes a strong solution and tends to the non-vacuum equilibrium state exponentially
in time.

According to the conclusions obtained in Theorems \ref{th2}-\ref{th1}, one finds that if we replaced the domain $\Omega=[0,1]$ by the whole space $\mathbb{R}$, at least for sufficiently smooth initial values, the phenomena on  vacuum states observed in \eqref{vanishvacuum} does not happen again. Actually, for the global regular solutions obtained in Theorems \ref{th2}-\ref{th1} to the Cauchy problem \eqref{e1.1}-\eqref{e1.3}, it is obvious that 
$$
 \inf_{x\in \mathbb{R}}\rho(t,x)=0 \quad \text{for any }\quad t\geq 0.
$$

\end{rk}

\subsection{Remarks on the asymptotic behavior of  $u$}

Moreover, in contrast to the classical theory in the case of  the constant viscosity \cite{HX1,mat}, one can  show that the $L^\infty$ norm of $u$ of the global regular solution we obtained does not decay to zero as time $t$ goes to infinity.
First, based on the  physical quantities introduced  in Remark \ref{r1},  we define  a solution class  as follows:
\begin{mydef}\label{d2}
Let $T>0$ be any constant. For   the Cauchy problem \eqref{e1.1}-\eqref{e1.3},  a classical solution     $(\rho,u)$ is said to be in $ D(T)$ if $(\rho,u)$  satisfies the following conditions:
\begin{equation*}\begin{split}
&(\textrm{A})\quad \text{Conservation of total mass:}\\
&\qquad \qquad   \qquad 0<m(0)=m(t)<\infty \ \text{for  any}  \ t\in [0,T];\\
&(\textrm{B})\quad \text{Conservation of momentum:}\\  
&\qquad \qquad  \qquad 0<|\mathbb{P}(0)|=|\mathbb{P}(t)|<\infty \ \text{for  any}  \ t\in [0,T];\\
&(\textrm{C})\quad \text{Finite kinetic energy:}\\
&\qquad \qquad   \qquad 0<E_k(t)<\infty  \ \text{for  any}  \ t\in [0,T].
\end{split}
\end{equation*}
\end{mydef}

Then one has:
\begin{thm}\label{th:2.20}
Assume 
$\gamma > 1$ and $ \delta\geq 0$.
 Then for   the Cauchy problem \eqref{e1.1}-\eqref{e1.3}, there is no classical solution $(\rho,u)\in D(\infty)$  satisfying 
\begin{equation}\label{eq:2.15}
\limsup_{t\rightarrow \infty} |u(t,\cdot)|_{\infty}=0.
\end{equation}

\end {thm}

According to  Theorems \ref{th2}-\ref{th:2.20} and Remark \ref{r1}, one shows that
\begin{cor}\label{th:2.20-c}
 Assume that   $|\mathbb{P}(0)|>0$.
Then  for   the Cauchy problem \eqref{e1.1}-\eqref{e1.3},  the   global  regular  solutions $(\rho,u)$ obtained in Theorems \ref{th2}-\ref{th1} do not 
satisfy \eqref{eq:2.15}.

\end {cor}

The rest of the paper is organized as follows:   \S $2$-\S$3$ are devoted to establishing the global well-posedness of regular solutions with large data stated in Theorems \ref{th2}-\ref{th1}.  We start with the reformulation of the original problem  \eqref{e1.1}-\eqref{e1.3}  as (\ref{e2.2}) ({\rm resp.} \eqref{e2.1}) in terms of the new variables for the case $0<\delta<1$ ({\rm resp.} $\delta=1$), and establish the local-in-time well-posedness of smooth solutions to the reformulated problems  under the assumption  that  the initial density  keeps positive for all point $x\in \mathbb{R}$  but decays to zero as $|x|\rightarrow \infty$ in \S 2.    The local smooth solutions to  (\ref{e2.2}) and (\ref{e2.1})  are  shown to be global ones  in 
\S 3 by deriving  the global-in-time a priori estimates  through energy methods. Here,  we have employed some  arguments due to  
Bresch-Desjardins \cite{bd6, bdvv} and  Bresch-Desjardins-Lin \cite{bd2}   to deal with the strong degeneracy of the density-dependent viscosity   $\mu(\rho)=\alpha \rho^\delta$, which include the well-known  B-D entropy estimates (see Lemma \ref{l4.2}) and the effective velocity (see Lemma \ref{l4.4}). Finally, in \S 4, we give the  proofs for  the non-existence theories of global regular solutions with $L^\infty$ decay on $u$ shown in Theorem \ref{th:2.20} and Corollary \ref{th:2.20-c}.  Furthermore,  we give several appendixes to list some lemmas that are frequently used in our proof,   and review some related  local-in-time well-posedness theories  for multi-dimensional  flows of degenerate viscosities.


\section{Local-in-time well-posedness}

This section will be  devoted to providing the local-in-time well-posedenss of regular solutions with far field vacuum  to the Cauchy problem \eqref{e1.1}-\eqref{e1.3} for the cases $0<\delta\le 1$. 

\subsection{\textbf{Case 1:} $0<\delta<1$}

For the case $0<\delta<1$, the corresponding local-in-time well-posedenss can be achieved by the following four steps.
\subsubsection{\rm{Reformulation}}

Via  introducing the following new variables \begin{equation}\label{tr}
 \phi=\frac{A\gamma}{\gamma-1}\rho^{\gamma-1}, \quad \psi=\frac{\delta}{\delta-1}(\rho^{\delta-1})_x,
 \end{equation}
 the Cauchy problem  \eqref{e1.1}-\eqref{e1.3} can be rewritten as 
\begin{equation}\label{e2.2}
\left\{
\begin{aligned}
&\ \phi_t+u\phi_x+(\gamma-1)\phi u_x=0,\\[6pt]
&\ u_t+u u_x+\phi_x-a\alpha\phi^{2e}u_{xx}=\alpha\psi u_x,\\[6pt]
&\ \psi_t+(u\psi)_x+(\delta-1)\psi u_x+a\delta \phi^{2e} u_{xx}=0,\\[6pt]
&\ (\phi,u,\psi)|_{t=0}=(\phi_0, u_0,\psi_0)\\[6pt]
=&\Big(\frac{A\gamma}{\gamma-1}\rho^{\gamma-1}_0(x),u_0(x),\frac{\delta}{\delta-1}(\rho^{\delta-1}_0(x))_x\Big)\quad \text{for} \quad x\in \mathbb R,\\[6pt]
&\ (\phi,u,\psi)\to (0,0,0) \quad \text{as}\quad  |x|\to \infty  \quad \text{for} \quad  t\geq0,
\end{aligned}
\right.
\end{equation}
where the constants $a$ and $e$ are given by
$$a=\Big(\frac{A\gamma}{\gamma-1}\Big)^{\frac{1-\delta}{\gamma-1}}\quad \text{and} \quad e=\frac{\delta-1}{2(\gamma-1)}<0.$$

\subsubsection{\rm{Discussion on $\nabla\rho_0^{\frac{\delta-1}{2}}\in L^4$ in \cite{zz2}}}

In \cite{zz2}, for the  3-D degenerate system \eqref{eq:1.1}-\eqref{10000}, via introducing the following new variables 
$$
 \phi=\frac{A\gamma}{\gamma-1}\rho^{\gamma-1}, \quad \psi=\frac{\delta}{\delta-1}\nabla\rho^{\delta-1},
$$
Xin-Zhu studied  the following reformulated problem:
\begin{equation}\label{rfp}
\left\{
\begin{aligned}
&\  \phi_t+u\cdot \nabla \phi+(\gamma-1)\phi \text{div} u=0,\\[6pt]
&\ u_t+u\cdot\nabla u +\nabla \phi+a\phi^{2e}Lu=\psi \cdot Q(u),\\[6pt]
&\ \psi_t+\nabla (u\cdot \psi)+(\delta-1)\psi\text{div} u +a\delta \phi^{2e}\nabla \text{div} u=0,\\[6pt]
&\ (\phi, u,\psi)|_{t=0}=(\phi_0, u_0,\psi_0)\\[6pt]
=&\Big(\frac{A\gamma}{\gamma-1} \rho^{\gamma-1}_0(x), u_0(x), \frac{\delta}{\delta-1}\nabla \rho^{\delta-1}_0(x)\Big)\quad \text{for} \quad x\in \mathbb{R}^3,\\[6pt]
&\ (\phi, u,\psi)\rightarrow (0,0,0)\quad \text{as}\quad |x|\rightarrow \infty\quad \text{for} \quad t \geq 0.
\end{aligned}
\right.
\end{equation}

For the above  problem,  the local-in-time existence   of the unique regular solution with far field vacuum  has been proven  in \cite{zz2} (see Lemma \ref{apd1} in Appendix C). It is worth pointing out that in the initial assumptions of Lemma \ref{apd1}, Xin-Zhu required that  
\begin{equation}\label{icd}
\nabla\rho_0^{\frac{\delta-1}{2}}\in L^4,
\end{equation}
which, actually,  is only used in the initial data's approximation process from the non-vacuum flow to the flow with far field vacuum, and does not appear in the  corresponding energy estimates process in the proof for the local existence of the regular solution shown in \cite{zz2}. Therefore, in Lemma \ref{apd1}, the  regularity of the quantity $\nabla\rho^{\frac{\delta-1}{2}}$ in positive time  has not been mentioned.


Actually, in Section 3.6 (page 136) of \cite{zz2}, in order to approximate  the flow with far field vacuum by the non-vacuum flow, for any $\eta\in (0,1)$,  one sets
$$\phi^\eta_{ 0}=\phi_0+\eta,\quad \psi^\eta_{ 0}=\frac{a\delta}{\delta-1} \nabla (\phi_0+\eta)^{2e},\quad h^\eta_0=(\phi_0+\eta)^{2e}.$$
Then the initial compatibility conditions of the perturbed initial data can be given as
\begin{equation*}
\begin{cases}
\displaystyle
\nabla u_0=(\phi_0+\eta)^{-e}g^\eta_1,\quad a Lu_0=(\phi_0+\eta)^{-2e}g^\eta_2,\\[12pt]
\displaystyle
\nabla \Big(a(\phi_0+\eta)^{2e}Lu_0\Big)=(\phi_0+\eta)^{-e}g^\eta_3,
\end{cases}
\end{equation*}
where $g^{\eta}_i$ ($i=1,2,3$) are given as
\begin{equation}\label{leitong}\begin{cases}
\displaystyle
g^\eta_1=\frac{\phi^{-e}_0}{(\phi_0+\eta)^{-e}}g_1,\quad g^\eta_2=\frac{\phi^{-2e}_0}{(\phi_0+\eta)^{-2e}}g_2,\\[8pt]
\displaystyle
g^\eta_3=\frac{\phi^{-3e}_0}{(\phi_0+\eta)^{-3e}}\Big(g_3-\frac{a\eta\nabla \phi^{2e}_0}{\phi_0+\eta}\phi^e_0Lu_0\Big).
\end{cases}
\end{equation}
 According to the definition of $\psi_0^\eta$, one has 
\begin{equation}\label{zp}
\begin{split}
\psi_0^\eta=&\frac{a\delta}{\delta-1}\Big(\frac{\phi_0}{\phi_0+\eta}\Big)^{1-2e}\nabla\phi_0^{2e},\\
\nabla\psi_0^\eta=&\frac{a\delta}{\delta-1}\Big(\Big(\frac{\phi_0}{\phi_0+\eta}\Big)^{1-2e}\nabla^2\phi_0^{2e}+\frac{(1-2e)\eta\nabla\phi_0\otimes \nabla\phi_0^{2e}}{(\phi_0+\eta)^2}\Big(\frac{\phi_0}{\phi_0+\eta}\Big)^{-2e}\Big)\\
=&\frac{a\delta}{\delta-1}\Big(\Big(\frac{\phi_0}{\phi_0+\eta}\Big)^{1-2e}\Big(\nabla^2\phi_0^{2e}+\frac{2(1-2e)}{e}\frac{\eta}{\phi_0+\eta}\nabla\phi_0^{e}\otimes \nabla\phi_0^{e}\Big)\Big),\\
\nabla^2\psi_0^\eta=&\frac{a\delta}{\delta-1}\Big(\Big(\frac{\phi_0}{\phi_0+\eta}\Big)^{1-2e}\Big(\nabla^3\phi_0^{2e}+\frac{1-2e}{e}\frac{\eta\phi_0^{-e}}{\phi_0+\eta}\nabla\phi_0^e\otimes\nabla^2\phi_0^{2e}\Big)\\
&+\frac{2(1-2e)}{e}\Big(\frac{-\eta\phi_0^{1-e}}{e(\phi_0+\eta)^2}(\nabla\phi_0^e)^3+\frac{2\eta}{\phi_0+\eta}\nabla\phi_0^e\otimes\nabla^2\phi_0^e\Big)\Big).
\end{split}
\end{equation}
From  the formula \eqref{zp}, one can find that the condition  \eqref{icd} is only used  to guarantee the uniform boundedness of $ |\psi^\eta_{0}|_{D^1_*}$ with respect to $\eta$. Then it follows from the initial assumptions of Lemma \ref{apd1} that, there exists a $\eta_{1}>0$ such that for any  $0<\eta\leq \eta_{1}$, 
\begin{equation*}\label{co-verfify}
\begin{split}
1+\eta+\|\phi^\eta_{0}-\eta\|_{3}+|\psi^\eta_{0}|_{D^1_*\cap D^2}+\|u_0\|_{3}+|g^\eta_1|_2+|g^\eta_2|_2+|g^\eta_3|_2\\
+\|(h^{\eta}_0)^{-1}\|_{L^\infty\cap D^{1,6} \cap D^{2,3} \cap D^3}+\|\nabla h^\eta_0 /h^{\eta}_0\|_{L^\infty \cap L^6\cap D^{1,3}\cap D^2} & \leq c_0,
\end{split}
\end{equation*}
where $c_0$ is a positive constant independent of $\eta$. 
Based on such kind of initial approximation process,  the Cauchy  problem \eqref{rfp} can be solved locally in time    by studying the corresponding Cauchy  problem away from the vacuum and then  passing to the limit as $\eta\to 0$. The details can be found in \cite{zz2}.

For 1-D case, the corresponding initial data's approximation process from the non-vacuum flow to the flow with far field vacuum   can also  be achieved  from the above process with some minor modifications. Similarly, for any $\eta\in (0,1)$,  one sets
$$\phi^\eta_{ 0}=\phi_0+\eta,\quad \psi^\eta_{ 0}=\frac{a\delta}{\delta-1} \Big((\phi_0+\eta)^{2e}\Big)_x,\quad h^\eta_0=(\phi_0+\eta)^{2e}.$$
Then the initial compatibility conditions of the perturbed initial data can be given as
\begin{equation*}
\displaystyle
\partial_xu_0=(\phi_0+\eta)^{-e}g^\eta_1,\quad  a\alpha \partial^2_xu_0=(\phi_0+\eta)^{-2e}g^\eta_2,
\end{equation*}
where $g^{\eta}_i$ ($i=1,2$) are given in the first line of \eqref{leitong}.
Then it follows from \eqref{e1.7}-\eqref{e1.8} that there also exists a $\eta_1>0$ such that for any $0<\eta<\eta_1$, 
\begin{equation*}\label{st}
\begin{split}
1+\eta+\|\phi_0^\eta-\eta\|_2+\|\psi_0^\eta\|_1+\|u_0\|_2+|g_1^\eta|_2 &\\
+|g_2^\eta|_2+\|(h^{\eta}_0)^{-1}\|_{L^\infty\cap D^{1} \cap D^{2} }+\|(\ln h^\eta_0)_x\|_{1} &\le c_0.
\end{split}
\end{equation*}
 However, compared with the 3D case in \cite{zz2}, it is worth pointing out  that the following initial assumption 
\begin{equation}\label{asu}
\partial_x\rho_0^{\frac{\delta-1}{2}}\in L^4
\end{equation} 
is no longer required  for the current 1D problem. Actually, one easily has 
\begin{equation*}
\partial_x \rho_0^{\frac{\delta-1}{2}}=\frac12\rho_0^{\frac{1-\delta}{2}}\partial_x \rho_0^{\delta-1},
\end{equation*}
 which,  along with the fact $\partial_x\rho_0^{\delta-1} \in H^1$, 
 and  the Sobolev embedding theorem, implies that  \eqref{asu} holds automatically.

\subsubsection{\rm{Local-in-time well-posedness}}

Since the mathematical structure of the system considered in problem \eqref{e2.2} is exactly the same as the one in problem  \eqref{rfp}, then based on the discussion in Section 2.1.2 and Lemma \ref{apd1}, one can obtain the following  local-in-time  well-posedness theory in 1-D space.

\begin{lem}\label{l2}
Assume that $0<\delta<1$ and  $\gamma>1$. If the initial data $(\rho_0, u_0)$ satisfy \eqref{e1.7}-\eqref{e1.8} except $\rho_0\in L^1$,
then there exist a time $T_*>0$ and a unique regular solution $(\rho, u)$ in $[0,T_*]\times \mathbb R$  to the Cauchy problem \eqref{e1.1}-\eqref{e1.3} satisfying \eqref{er1} with $T$ replaced by $T_*$ except $\rho\in C([0,T_*];L^1)$.
\end{lem}

\subsubsection{\rm{$\rho\in C([0,T_*];L^1$)}}\label{verf}

It should be emphasized that the initial assumption $\rho_0\in L^1$ in Theorem \ref{th2} is only used in the proof of global-in-time existence of the unique regular solution satisfying \eqref{er1}, which is not required  in the proof for the local-in-time existence of the general regular solution  of  \cite{zz2}. Thus, in the following lemma, we  show  that actually, the solution obtained in Lemma 2.1 still satisfies:
\begin{lem}\label{vsi}Let $(\rho, u)$ in $[0,T_*]\times \mathbb R$  be the solution  to the Cauchy problem \eqref{e1.1}-\eqref{e1.3} obtained in Lemma 2.1. Then
\begin{equation}\label{kakkkk}\rho\in C([0,T_*];L^1) \quad \text{if} \quad \rho_0\in L^1 \quad \text{additionally}. 
\end{equation}
\end{lem}

\begin{proof}

First,  one proves  that
\begin{equation}\label{adiqq}
\rho\in L^\infty([0,T_*];L^1).
\end{equation}
 For this purpose,   let $f: \mathbb R^{+}\to \mathbb R$ satisfy
\begin{align*}
f(s)=\begin{cases}
1, & s \in [0,\frac{1}{2}]\\
\text{non-negative polynomial}, & s \in [\frac{1}{2}, 1]\\
e^{-s}, & s \in [1,\infty)
\end{cases}
\end{align*}
such that $f \in C^2$. Then there exists a generic constant $C_* >0$ such that $
|f'(s)| \le C_* f(s)$.
Define, for any $R>0$, $
f_R(x)=f(\frac{|x|}{R})$.
Then, according to Lemma \ref{l2},   one can obtain  that for any given $R>0$,
\begin{align*}
\int \Big((|\rho_t|+|\rho_x u|+|\rho u_x|)f_R(x)+\rho f_R(x)+|\rho u \partial_x f_R(x)|\Big) \text{d}x<& \infty.
\end{align*}
Multiplying $\eqref{e1.1}_1$ by $f_R(x)$ and integrating over $\mathbb R$, one has 
\begin{equation*}
\begin{split}
\frac{\text{d}}{\text{d}t}\int \rho f_R(x)\text{d}x =&-\int (\rho_x u+\rho u_x)f_R(x)\text{d}x\\
=&\int \rho u f'_R(x)\text{d}x\leq C|u|_\infty \int \rho f_R(x)\text{d}x
\end{split}
\end{equation*}
for some positive constant $C$ depending only on the initial data, $A$, $\gamma$, $\delta$, $C_*$  and $T_*$ but independent of $R$, which, along with the Gronwall inequality, implies that 
\begin{equation}
\int \rho f_R(x) \text{d}x\leq C\quad \text{for}\quad  t\in [0,T_*].
\end{equation}
Since $\rho f_R(x)\to \rho$ almost everywhere as $R\to \infty$, then it follows from the Fatou lemma (see Lemma \ref{Fatou} in Appendix A) that 
\begin{equation}
\int \rho\text{d}x\leq \liminf\limits_{R\to \infty}\int \rho f_R(x) \text{d}x\leq C\quad \text{for}\quad  t\in [0,T_*].
\end{equation}

Second,  according to $\eqref{e1.1}_1$,  one has 
\begin{equation}\label{zzx}
|\rho_t|_1\leq  C(|\rho_x u|_1+|\rho u_x|_1)\leq C(|\psi|_2|u|_2|\rho|_\infty^{2-\delta}+|\rho|_2|u_x|_2)\leq C,
\end{equation}
 which,  along with \eqref{adiqq} and  the Sobolev embedding theorem, implies that \eqref{kakkkk} holds.

The proof of Lemma \ref{vsi} is complete.
\end{proof}

\subsection{\textbf{Case 2:} $\delta=1$}

For the case $\delta=1$, the corresponding local-in-time well-posedenss can be achieved by the following three steps.

\subsubsection{\rm{Reformulation}}
Via introducing the following  new variables \begin{equation}\label{bre}
\phi=\frac{A\gamma}{\gamma-1}\rho^{\gamma-1}, \quad  \psi=\frac{1}{\gamma-1}\frac{\phi_x}{\phi}=(\ln\rho)_x,\end{equation}
 the Cauchy problem \eqref{e1.1}-\eqref{e1.3} can be rewritten as 
\begin{equation}\label{e2.1}
\left\{
\begin{aligned}
&\ \phi_t+u\phi_x+(\gamma-1)\phi u_x=0,\\[6pt]
&\ u_t+u u_x+\phi_x-\alpha u_{xx}=\alpha \psi u_x,\\[6pt]
&\ \psi_t+(u\psi)_x+u_{xx}=0,\\[6pt]
&\ (\phi,u,\psi)|_{t=0}=(\phi_0, u_0,\psi_0)\\[6pt]
=&\Big(\frac{A\gamma}{\gamma-1}\rho^{\gamma-1}_0(x),u_0(x),(\ln\rho_0(x))_x\Big)\quad \text{for} \quad x\in \mathbb R,\\[6pt]
&\ (\phi,u,\psi)\to (0,0,0) \quad \text{as}\quad |x|\to \infty\quad \text{for} \quad t\geq 0.
\end{aligned}
\right.
\end{equation}

%

\subsubsection{\rm{Local-in-time existence}}

In \cite{sz3}, for the  2-D degenerate system \eqref{eq:1.1}-\eqref{10000}, via introducing the following new variables  
$$
\tilde\phi=\rho^{\frac{\gamma-1}{2}}, \quad \psi=\frac{2}{\gamma-1}\frac{\nabla \tilde\phi}{\tilde\phi}=\nabla \ln \rho,
$$
Li-Pan-Zhu studied  the following reformulated problem:
\begin{equation}\label{ref2}
\left\{
\begin{aligned}
&\ \tilde{\phi_t}+u\cdot\nabla\tilde\phi+\frac{\gamma-1}{2}\tilde\phi \diver u=0,\\[2pt]
&\ u_t+u\cdot\nabla u+\frac{2A\gamma}{\gamma-1}\tilde\phi \nabla \tilde\phi+Lu=\psi \cdot Q(u),\\[2pt]
&\ \psi_t+\nabla(u\cdot \psi)+\nabla\diver u=0,\\[2pt]
&(\tilde\phi,u,\psi)|_{t=0}=(\tilde\phi_0, u_0,\psi_0)\\[2pt]
=&\Big(\rho^{\frac{\gamma-1}{2}}_0(x),u_0(x),\nabla \ln\rho_0(x)\Big)\quad \text{for} \quad x\in \mathbb R^2,\\[2pt]
&\ (\tilde\phi,u,\psi)\to (0,0,0) \quad \text{as}\quad |x|\to \infty\quad \text{for} \quad t\geq 0.
\end{aligned}
\right.
\end{equation}
Since the mathematical structure of the system considered in problem \eqref{e2.1} is exactly the same as the one in problem  \eqref{ref2}, then based on Lemma \ref{apd2}, one can obtain the following  local-in-time  well-posedness theory in 1-D space.

\begin{lem}\label{l1}
Assume that $\delta=1$ and  $\gamma>1$. If  the initial data $(\rho_0, u_0)$ satisfy \eqref{id1} except $\rho_0\in L^1$,
then there exist a time $T_*>0$ and a unique regular solution $(\rho, u)$ in  $[0,T_*]\times \mathbb R$ to the Cauchy problem \eqref{e1.1}-\eqref{e1.3} satisfying \eqref{er2} with $T$ replaced by $T_*$ except $\rho\in C([0,T_*];L^1)$.
\end{lem}

\subsubsection{\rm{$\rho\in C([0,T_*];L^1)$}} 
As mentioned in Section \ref{verf}, one  also needs to show  that 
\begin{lem}\label{vsikkk}Let $(\rho, u)$ in $[0,T_*]\times \mathbb R$  be the solution  to the Cauchy problem \eqref{e1.1}-\eqref{e1.3} obtained in Lemma 2.3. Then 
\begin{equation}\label{kakkkkmmm}\rho\in C([0,T_*];L^1) \quad \text{if} \quad \rho_0\in L^1 \quad \text{additionally}. 
\end{equation}
\end{lem}
The proof of this lemma is similar to that of  Lemma \ref{vsi}. Here we omit it.
\section{Global-in-time well-posedness}

This section is devoted  to proving Theorems \ref{th2}-\ref{th1}, and we always  assume  $0<\delta\le1$ and  $\gamma>1$.
 Let $T>0$ be some time and $(\rho, u)$ be  regular solutions to the  problem \eqref{e1.1}-\eqref{e1.3} in $[0,T]\times \mathbb R$ obtained in Lemmass \ref{l2}-\ref{vsi} and Lemmas \ref{l1}-\ref{vsikkk}.
The main aim in the rest of this section is to  establish the global-in-time  a priori estimates for these solutions.   Hereinafter, we denote  $C_0$ {\rm (resp. $C$)} a  generic positive constant depending only on $(\rho_0,u_0, A, \gamma,\alpha,\delta)$  {\rm (resp. $(C_0,T)$)}, which may be different  from line to line. 
\par\smallskip


\subsection{The $L^\infty$ estimate of  $\rho$}
We consider the upper bound of the mass density $\rho$ in $[0,T]\times \mathbb R$.
First, the standard energy estimates yields that
\begin{lem}\label{l4.1}
For any $T>0$, it holds that
\begin{equation*}\label{e-1.1}
\int \left(\frac12\rho u^2+\frac{A}{\gamma-1}\rho^\gamma\right)(t,\cdot){\rm{d}}x+\alpha\int_0^t\int \rho^{\delta} u_x^2 {\rm{d}}x{\rm{d}}s\leq C_0 \quad {\rm{for}}\ \ 0\leq t\leq T.
\end{equation*}
\end{lem}

Second, we give the well-known B-D entropy estimate.
\begin{lem}[\cite{bd}]\label{l4.2}
For any $T>0$, it holds that
\begin{equation*}\label{e-1.2}
\int\left(\frac12\rho\left|u+\varphi(\rho)_x\right|^2+\frac{A}{\gamma-1}\rho^\gamma\right)(t,\cdot){\rm{d}}x+\alpha A\gamma\int_0^t\int \rho^{\gamma+\delta-3}\rho_x^2{\rm{d}}x{\rm{d}}s\leq C_0
\end{equation*}
for $ 0\leq t\leq T$, where $\varphi'(\rho)=\frac{\mu(\rho)}{\rho^2}$.
\end{lem}
The conclusions obtained in Lemmas \ref{l4.1}-\ref{l4.2} are classical, and the corresponding  proofs can be found in Appendix B.

Next we show the regular solution $(\rho,u)$ keeps the conservation of total mass.
\begin{lem}\label{ms}
For any $T>0$, it holds that
$$m(t)=m(0)  \quad {\rm{for}}\ \ 0\leq t\leq T.$$
\end{lem}

\begin{proof}
It follows from $\eqref{e1.1}_1$ that 
\begin{equation*}
    \frac{\text{d}}{\text{d}t}\int \rho(t,\cdot)\text{d}x=-\int (\rho u)_x(t,\cdot)\text{d}x=0,
\end{equation*}
where one has used fact that 
$$\rho\in C([0,T];L^1)\quad \text{and}\quad \rho u(t,\cdot)\in W^{1,1}(\mathbb R).$$
The proof of Lemma \ref{ms} is complete.
\end{proof}

Now we are ready to give the uniform  upper bound of the density.

\begin{lem}\label{l4.3}
For any $T>0$, it holds that
\begin{equation*}\label{e4.3}
|\rho(t,\cdot)|_\infty\leq C_0 \quad {\rm{for}}\ \ 0\leq t\leq T.
\end{equation*}
\end{lem}

\begin{proof}

First, integrating the  equation $\eqref{e1.1}_2$ over $(-\infty, x]$ with respect to $x$, one has 
\begin{equation}\label{cc}
\frac{\text{d}}{\text{d}t}\int_{-\infty}^x \rho u(t, z){\rm{d}}z+\rho u^2+A\rho^\gamma=\alpha\rho^\delta u_{x},
\end{equation}
where one has used the far field behavior $\eqref{e1.3}_2$. 

Denote $\displaystyle \xi(t, x)=\int_{-\infty}^x \rho u(t, z){\rm{d}}z$. Then it follows from  \eqref{cc} that  
\begin{equation}\label{e-1.7}
\xi_t+u\xi_x+\alpha\rho^{\delta-1}(\rho_t+u\rho_x) +A\rho^\gamma=0.
\end{equation}

Second, let  $y(t;x)$ be the solution of the following problem
\begin{equation}
\left\{
\begin{aligned}
&\frac{\text{d}y(t; x)}{\text{d}t}=u(t, y(t;x)),\\
& y(0; x)=x.
\end{aligned}
\right.
\end{equation}
It follows from \eqref{e-1.7} that 
\begin{equation*}
\frac{\text{d}}{\text{d}t}\left(\xi+\frac{\alpha}{\delta}\rho^{\delta}\right)(t, y(t; x))\leq 0,
\end{equation*}
which, along with the fact 
\begin{equation}\label{gg}\xi(t, y(t;x))\leq |\sqrt\rho u(t,\cdot)|_{2}|\rho(t,\cdot)|_{1}^{\frac12}\leq C_0\quad {\rm{for}}\ \ 0\leq t\leq T,
\end{equation}
implies that 
\begin{equation}\label{gg1}
\left(\xi+\frac{\alpha}{\delta}\rho^{\delta}\right)(t, y(t; x))\leq \xi(0, y(0; x))+
\frac{\alpha}{\delta}\rho_0^{\delta}(x)\leq C_0.
\end{equation}

Finally,  according to \eqref{gg}-\eqref{gg1}, one can obtain that 
\begin{equation*}
|\rho^{\delta}(t,\cdot)|_\infty\leq C_0 \quad {\rm{for}}\ \ 0\leq t\leq T.
\end{equation*}
The proof of Lemma \ref{l4.3} is complete.

\end{proof}

\subsection{The $L^2$ estimate of $u$}
We consider the $L^2$ estimate of $u$. For this purpose,  one first needs to show  the following several  auxiliary lemmas. The first one is on the $L^q$ estimate of  $\rho^{\frac1q}u$ for any  $2\le q< \infty$.



\begin{lem} \label{lma}
Assume  $2\le q< \infty$. Then for any  $T>0$, it holds that 
\begin{equation*}
|\rho^{\frac{1}{q}} u(t, \cdot)|_q \leq C \quad {\rm{for}}\ \ 0\leq t\leq T,
\end{equation*}
where the constant $C$ depends on $q$.
\end{lem}
\begin{proof}
First, multiplying $\eqref{e1.1}_2$ by $|u|^{p}u$ $(p\ge2)$ and integrating over $\mathbb R$, one has 
\begin{equation}\label{add}
\begin{split}
&\frac{1}{p+2} \frac{\text{d}}{\text{d}t}\big|\rho^{\frac{1}{p+2}} u\big|_{p+2}^{p+2}  +\alpha (p+1) \big|\rho^{\frac{\delta}{2}}u^{\frac{p}{2}} u_x\big|_2^2\\
=& A(p+1)\int  \rho^{\gamma} |u|^p u_x{\rm{d}}x\\
\le & \frac{\alpha(p+1)}{2} \int \rho^{\delta}|u|^p u_x^2{\rm{d}}x+\frac{A^2(p+1)}{2\alpha} \int \rho^{2\gamma-\delta} |u|^p{\rm{d}}x.
\end{split}
\end{equation}

When $p=2$, it follows from \eqref{add}, Lemmas \ref{l4.1} and \ref{ms}-\ref{l4.3} that
\begin{equation}\label{addd}
\begin{split}
&\frac14 \frac{\text{d}}{\text{d}t}\big|\rho^{\frac14} u\big|_{4}^{4}  +\frac{3\alpha}{2} \big|\rho^{\frac{\delta}{2}}u u_x\big|_2^2\\
\le & C|\rho|_\infty^{2\gamma-\delta-1}|\rho^{\frac14} u|_4|\rho^{\frac12} u|_2|\rho|_1^{\frac14}\leq C(1+|\rho^{\frac14} u|_4^4).
\end{split}
\end{equation}

When $p>2$, it follows from \eqref{add}, Lemmas \ref{l4.1} and \ref{ms}-\ref{l4.3} that
\begin{equation}\label{adddd}
\begin{split}
&\frac{1}{p+2} \frac{\text{d}}{\text{d}t}\big|\rho^{\frac{1}{p+2}} u\big|_{p+2}^{p+2}  +\frac{\alpha(p+1)}{2} \big|\rho^{\frac{\delta}{2}}u^{\frac{p}{2}} u_x\big|_2^2\\
\le& C|\rho|_\infty^{2\gamma-\delta-1}|\rho^{\frac{p-2}{p}}u^{\frac{(p-2)(p+2)}{p}}|_{\frac{p}{p-2}}|\rho^{\frac 2p}u^{\frac 4p}|_{\frac p2}\\
\le &C | \rho|_{\infty}^{2\gamma-\delta-1} |\rho^{\frac{1}{p+2}} u|_{p+2}^{\frac{(p-2)(p+2)}{p}} |\rho^{\frac12} u |_2^{\frac{4}{p}}\\
\le & C |\rho^{\frac{1}{p+2}} u|_{p+2}^{\frac{(p-2)(p+2)}{p}} 
\le  C\big(1+ |\rho^{\frac{1}{p+2}} u|_{p+2}^{p+2}\big).
\end{split}
\end{equation}

Second, integrating \eqref{addd} and \eqref{adddd} over $[0,t]$, one can obtain that  for any $p\ge 2$,  
\begin{equation*}
\begin{split}
&\big|\rho^{\frac{1}{p+2}} u\big|_{p+2}^{p+2} +\int_0^t \big|\rho^{\frac{\delta}{2}}u^{\frac{p}{2}} u_x\big|_2^2\text{d}s\\
\leq &\big|\rho_0^{\frac{1}{p+2}} u_0\big|_{p+2}^{p+2}+C\Big(t+\int_0^t \big|\rho^{\frac{1}{p+2}} u\big|_{p+2}^{p+2}\text{d}s\Big),
\end{split}
\end{equation*}
which, along with the Gronwall inequality, yields that
\begin{equation}\label{polo}
\begin{split}
&\big|\rho^{\frac{1}{p+2}} u(t,\cdot)\big|_{p+2}^{p+2} +\int_0^t \big|\rho^{\frac{\delta}{2}}u^{\frac{p}{2}} u_x(s,\cdot)\big|_2^2 {\rm{d}s} \\
\le&C \Big(1+\big|\rho_0^{\frac{1}{p+2}} u_0\big|_{p+2}^{p+2}\Big)
\le C\Big(1+|\rho_0|_1 |u_0|_{\infty}^{p+2}\Big)\le C\quad {\rm{for}}\ \ 0\leq t\leq T.
\end{split}
\end{equation}
Finally, the desired conclusion can be achieved  from \eqref{polo} and Lemma \ref{l4.1}. 

The proof of Lemma \ref{lma} is complete.
\end{proof}

The following lemma gives the estimate of $\int_0^t |\rho^{\iota}u(s,\cdot)|_\infty {\rm{d}}s$ $(\frac12\le \iota< 1)$, which plays a crucial role in obtaining the $L^\infty$ estimate of the so-called  effective velocity $v$ (see \eqref{ev}).

\begin{lem}\label{cru}
Assume 
\begin{equation}\label{zuizhong}
\iota=\frac12\quad \rm{for}\quad  0<\delta<1;\quad    \frac12<\iota<1\quad \rm{for}\quad  \delta=1.
\end{equation}
Then for any  $T>0$, it holds that 
\begin{equation*}
\int_0^t |\rho^{\iota}u(s,\cdot)|_{\infty} {\rm{d}}s\leq C \quad {\rm{for}}\ \ 0\leq t\leq T.
\end{equation*}
\end{lem}

\begin{proof}
{\textbf{Case 1:}} $0<\delta<1$. First, it follows  direct calculations that 
\begin{equation*}\label{ne1}
\partial_x(\rho^{\frac12} u)=\frac12\rho^{\delta-\frac32}\rho_x\rho^{\frac{1-\delta}{2}}\rho^{\frac{1-\delta}{2}} u+\rho^{\frac{1-\delta}{2}}\rho^{\frac{\delta}{2}}u_x.
\end{equation*}

According to Lemmas \ref{l4.1}-\ref{ms} and \ref{lma}, one has 
\begin{equation*}\label{ne2}
|\rho^{\delta-\frac32}\rho_x|_2+|\rho^{\frac{1-\delta}{2}}u|_{\frac{2}{1-\delta}}+|\rho^{\frac{1-\delta}{2}}|_{\frac{2}{1-\delta}}+\int_0^t |\rho^{\frac{\delta}{2}}u_x|_2^2 {\rm{d}}s\leq C,
\end{equation*}
which, along with Lemma \ref{l4.3}, yields that 
\begin{equation}\label{sg}
\int_0^t |\partial_x(\rho^{\frac12}u)(s,\cdot)|_{\frac{2}{2-\delta}}^2\text{d}s\le C \quad {\rm{for}}\ \ 0\leq t\leq T.
\end{equation}

Second, notice that  $\frac{2}{2-\delta}\in (1,2)$, then by  the Sobolev embedding theorem (or see Lemma \ref{al1}), one has 
$$|\rho^{\frac12}u|_\infty\le C|\rho^{\frac12}u|_2^{1-\aleph}|\partial_x(\rho^{\frac12}u)|^\aleph_{\frac{2}{2-\delta}}\quad \text{with}\quad \aleph=\frac{1}{1+\delta}\in (0,1),$$
which, along with Lemma \ref{l4.1} and \eqref{sg}, yields that
$$\int_0^t |\rho^{\frac12}u(s,\cdot)|_{\infty} {\rm{d}}s\leq C \int_0^t\Big(1+|\partial_x(\rho^{\frac12}u)|^2_{\frac{2}{2-\delta}}\Big)(s,\cdot)\text{d}s\leq C \quad {\rm{for}}\ \ 0\leq t\leq T.$$

{\textbf{Case 2:}} $\delta=1$.  First, it follows  direct calculations that 
\begin{equation*}
\partial_x(\rho^\iota u)=\iota\rho^{-\frac12}\rho_x\rho^{\iota-\frac12} u+\rho^{\iota-\frac12}\rho^{\frac12} u_x.
\end{equation*}
According to Lemmas \ref{l4.1}-\ref{ms}, \ref{lma} and $\frac12<\iota<1$, one has 
\begin{equation*}
|\rho^{-\frac12}\rho_x|_2+|\rho^{\iota-\frac12}u|_{\frac{2}{2\iota-1}}+|\rho^{\iota-\frac12}|_{\frac{2}{2\iota-1}}+\int_0^t |\rho^{\frac{1}{2}}u_x|_2^2 {\rm{d}}s\leq C,
\end{equation*}
which yields that 
\begin{equation}\label{sgg}
\int_0^t |\partial_x(\rho^{\iota}u)(s,\cdot)|_{\frac{1}{\iota}}^2\text{d}s\le C \quad {\rm{for}}\ \ 0\leq t\leq T.
\end{equation}

Second, notice that $\frac{1}{\iota}\in(1,2)$, then by Sobolev embedding theorem (or see Lemma \ref{al1}), one has 
$$|\rho^{\iota}u|_\infty\le C|\rho^{\iota}u|_2^{1-\Xi}|\partial_x(\rho^{\iota}u)|^\Xi_{\frac{1}{\iota}}\quad \text{with}\quad \Xi=\frac{1}{3-2\iota}\in(0,1),$$
which, along with Lemmas \ref{l4.1}, \ref{l4.3} and \eqref{sgg}, yields that

$$\int_0^t |\rho^{\iota}u(s,\cdot)|_{\infty} {\rm{d}}s\leq C \int_0^t\Big(1+|\partial_x(\rho^{\iota}u)|^2_{\frac{1}{\iota}}\Big)(s,\cdot)\text{d}s\leq C \quad {\rm{for}}\ \ 0\leq t\leq T.$$

The proof of Lemma \ref{cru} is complete.
\end{proof}

Next we show the  $L^\infty$ estimate of the effective velocity:
\begin{equation}\label{ev}
 v=u+\varphi(\rho)_x=u+\alpha \rho^{\delta-2}\rho_x.
\end{equation}


\begin{lem}\label{l4.4}
Assume additionally 
\begin{equation}\label{zuizhong}
\gamma \geq \delta+ \frac{1}{2}\quad \rm{for}\quad  0<\delta<1;\quad   \gamma>\frac32\quad \rm{for}\quad  \delta=1.
\end{equation} Then for any  $T>0$, it holds that 
\begin{equation*}
|v(t,\cdot)|_{\infty}\leq C\quad {\rm{for}}\ \ 0\leq t\leq T.
\end{equation*}
\end{lem}

\begin{proof}
First, it follows from  the system $\eqref{e1.1}$ and the definition of $v$ that 
\begin{equation}\label{e-1.16}
v_t+uv_x+\frac{A\gamma}{\alpha}\rho^{\gamma-\delta}v-\frac{A\gamma}{\alpha}\rho^{\gamma-\delta}u=0.
\end{equation}
Here, the detail derivation of the above equation for $v$ can be found in Appendix B.

Via the  standard characteristic method, one can obtain 
\begin{equation*}
v=\left(v_0+\int_0^t \frac{A\gamma}{\alpha} \rho^{\gamma-\delta}u\  \exp \Big(\int_0^s \frac{A\gamma}{\alpha} \rho^{\gamma-\delta}{\rm{d}}\tau\Big){\rm{d}}s\right)\exp\left(-\int_0^t \frac{A\gamma}{\alpha} \rho^{\gamma-\delta} {\rm{d}}s\right).
\end{equation*}
Then according to Lemmas \ref{l4.3} and  \ref{cru}, one has 
\begin{equation}\label{e-1.17}
\begin{split}
|v|_\infty\le & C\Big(|v_0|_\infty+\int_0^t|\rho^{\gamma-\delta}u|_\infty{\rm{d}}s\Big)\\
\le &C\Big(1+\int_0^t |\rho|^{\gamma-\delta-\frac{1}{2}}_{\infty} |\rho^{\frac{1}{2}}u|_\infty{\rm{d}}s\Big)\\
\le &C \quad \text{for} \quad 0<\delta<1, \quad  \gamma\ge \delta+\frac12,\\
|v|_\infty\le & C\Big(|v_0|_\infty+\int_0^t|\rho^{\gamma-1}u|_\infty{\rm{d}}s\Big)\\
\le & C\Big(1+\int_0^t |\rho|^{\gamma-1-\iota}_{\infty} |\rho^{\iota}u|_\infty{\rm{d}}s\Big)\\
\le &C \quad \text{for}\quad \delta=1,\quad  \gamma\ge 1+\iota>\frac32.
\end{split}
\end{equation}

 The  proof of Lemma \ref{l4.4} is complete.
\end{proof}

 Now we show the  $L^2$ estimate of velocity $u$.

\begin{lem}\label{l4.5}
Assume \eqref{zuizhong} additionally. Then for any $T>0$,  it holds that 
\begin{equation*}
|u(t,\cdot)|_2^2+\int_0^t| \rho^{\frac{\delta-1}{2}} u_x(s,\cdot)|_2^2{\rm{d}}s\leq C\quad {\rm{for}}\ \  0\le t\le T.
\end{equation*}
\end{lem}

\begin{proof}
{\textbf{Case 1:}} $0<\delta<1$.
 From $\eqref{e2.2}_2$, one has 
\begin{equation}\label{e4.14}
u_t+u u_x+A\gamma\rho^{\gamma-2}\rho_x=\alpha(\rho^{\delta-1}u_{x})_x+\alpha\rho^{\delta-2}\rho_x u_x.
\end{equation}
Multiplying \eqref{e4.14} by $u$ and integrating over $\mathbb R$, one has 
\begin{equation}\label{e4.15}
\begin{split}
&\frac12\frac{\text{d}}{\text{d}t} |u|_2^2 +\alpha|\rho^{\frac{\delta-1}{2}}u_x|_2^2\\
=& \int \big(-u u_{x}-A\gamma \rho^{\gamma-2}\rho_x +\alpha\rho^{\delta-2}\rho_x u_x\big) u{\rm{d}}x\triangleq  \sum_{i=1}^3 \text{J}_i.
\end{split}
\end{equation}

According to   H\"older's  inequality, Young's inequality and Lemmas \ref{l4.1}-\ref{l4.4}, one can obtain 
\begin{equation}\label{j13}
\begin{aligned}
\text{J}_1=&-\int u^2 u_{x}{\rm{d}}x=0,\\
\text{J}_2=&-A\gamma\int \rho^{\gamma-2}\rho_x u{\rm{d}}x=-A\gamma\int \rho^{\gamma-\delta}\rho^{\delta-2}\rho_x u{\rm{d}}x\\
=&-\frac{A\gamma}{\alpha}\int \rho^{\gamma-\delta}(v-u)u{\rm{d}}x\\
\leq& C\Big(|\sqrt\rho |_{2}|u|_{2}|v|_\infty|\rho|_{\infty}^{\gamma-\delta-\frac12}+|\rho|_{\infty}^{\gamma-\delta}|u|_{2}^2\Big)
\leq C|u|_{2}^2+C,\\
\text{J}_3=&\alpha\int \rho^{\delta-2}\rho_x u_x u{\rm{d}}x=\int \left(v-u\right)u_x u{\rm{d}}x\\
\leq &C|v|_{\infty}|\rho|_{\infty}^{\frac{1-\delta}{2}}|\rho^{\frac{\delta-1}{2}}u_x|_{2}|u|_{2}
\leq C|u|_{2}^2+\frac{\alpha}{2} |\rho^{\frac{\delta-1}{2}}u_x|_2^2.
\end{aligned}
\end{equation}
It follows from  \eqref{e4.15}-\eqref{j13} and  the Gronwall  inequality that 
\begin{equation*}\label{e4.16}
|u(t,\cdot)|_2^2+\int_0^t| \rho^{\frac{\delta-1}{2}} u_x(s,\cdot)|_2^2{\rm{d}}s\leq C\quad {\rm{for}}\ \ 0\leq t\leq T.
\end{equation*}

{\textbf{Case 2:}}  $\delta=1$.
Multiplying $\eqref{e2.1}_2$ by $u$ and integrating over $\mathbb R$, one has 
\begin{equation*}\label{e4.170}
\frac12\frac{\text{d}}{\text{d}t}|u|_2^2+\alpha | u_x|_2^2=\int \big(-u u_{x}-A\gamma\rho^{\gamma-2} \rho_x +\alpha\psi u_x \big)u{\rm{d}}x.
\end{equation*}
Via the similar argument for estimating $\text{J}_1$-$\text{J}_3$, one can obtain 
\begin{equation*}\label{e4.171}
|u(t,\cdot)|_2^2+\int_0^t| u_x(s,\cdot)|_2^2{\rm{d}}s\leq C\quad {\rm{for}}\ \ 0\leq t\leq T.
\end{equation*}

The proof of Lemma \ref{l4.5} is complete.
\end{proof}

\subsection{The  first order   estimates of $u$}
Now we are ready to consider the first order estimates of $u$.
\begin{lem}\label{l4.6}
Assume \eqref{zuizhong} additionally. Then for any $T>0$,  it holds that 
\begin{equation*}\label{e-1.21}
| \rho^{\frac{\delta-1}{2}}u_x(t,\cdot)|_2^2+\int_0^t\big( |u_t|_2^2+|\rho^{\delta-1}u_{xx}|_2^2+|u_x|_{\infty}\big)(s,\cdot){\rm{d}}s\leq C \quad {\rm{for}}\ \ 0\leq t\leq T.
 \end{equation*}
\end{lem}

\begin{proof}
\textbf{Case 1:} $0<\delta<1$.
Multiplying \eqref{e4.14} by $u_t$ and integrating over $\mathbb R$, one has 
\begin{equation}\label{e-1.20}
\begin{split}
&\frac{\alpha}{2}\frac{\text{d}}{\text{d}t}| \rho^{\frac{\delta-1}{2}}u_x|_2^2+| u_t|_2^2\\
=&\int\Big(\frac{\alpha}{2}(\rho^{\delta-1})_t u_x^2- \big(u u_x +A\gamma \rho^{\gamma-2}\rho_x -\alpha \rho^{\delta-2}\rho_x u_x\big) u_t\Big){\rm{d}}x\triangleq \sum_{i=4}^7 \text{J}_i.
\end{split}
\end{equation}

According to the equation  $\eqref{e1.1}_1$, H\"older's inequality, Lemma \ref{al1} and Lemmas \ref{l4.1}-\ref{l4.5}, one can obtain that 
\begin{equation}\label{j4}
\begin{split}
\text{J}_4=&\frac{\alpha}{2}\int (\rho^{\delta-1})_t u_x^2{\rm{d}}x=-\frac{\alpha}{2}\int \big((\rho^{\delta-1})_x u+(\delta-1)\rho^{\delta-1}u_x\big)u_x^2{\rm{d}}x\\
=&-\frac{\alpha(\delta-1)}{2}\int \rho^{\delta-2}\rho_x u u_x^2{\rm{d}}x-\frac{\alpha(\delta-1)}{2}\int \rho^{\delta-1}u_x^3{\rm{d}}x\\
=&-\frac{\delta-1}{2}\int (v-u)u u_x^2{\rm{d}}x-\frac{\alpha(\delta-1)}{2}\int \rho^{\delta-1}u_x^3{\rm{d}}x\\
=&-\frac{\delta-1}{2}\int v u u_x^2{\rm{d}}x+\frac{\delta-1}{2}\int u^2 u_x^2{\rm{d}}x-\frac{\alpha(\delta-1)}{2}\int \rho^{\delta-1}u_x^3{\rm{d}}x\\
\le &-\frac{\delta-1}{2}\int v u u_x^2{\rm{d}}x-\frac{\alpha(\delta-1)}{2}\int \rho^{\delta-1}u_x^3{\rm{d}}x\\
\le& C\big(|v|_{\infty}|u|_{\infty}|\rho|_{\infty}^{1-\delta}+|u_x|_{\infty}\big)| \rho^{\frac{\delta-1}{2}}u_x|_2^2\\
\le& C\big(1+|u_{xx}|_2^{\frac12}\big)|u_x|_2^{\frac12} | \rho^{\frac{\delta-1}{2}}u_x|_2^2.
\end{split}
\end{equation}

Then we  need to consider $|u_{xx}|_2$. According to \eqref{e4.14}, one has 
\begin{align*}
|u_{xx}|_{2}\leq& C|\rho|_\infty^{1-\delta}|\rho^{\delta-1}u_{xx}|_{2}\\
\leq &C\big(|u_t|_{2}+|u u_x|_{2}+|\rho^{\delta-2}\rho_x u_x|_{2}+|\rho^{\gamma-2}\rho_x|_{2}\big)\\
\leq &C\big(|u_t|_{2}+|u |_{2}|u_x|_{\infty}+|(v-u) u_x|_{2}+|\rho^{\gamma-\delta}(v-u)|_{2}\big)\\
\leq &C\big(|u_t|_{2}+|u|_{2} |u_x|_{2}^{\frac12}|u_{xx}|_{2}^{\frac12}+(|v|_{\infty}+|u|_{\infty})|u_x|_{2}\\
&+|\rho^{\gamma-\delta}|_{2}|v|_{\infty}+|\rho|_{\infty}^{\gamma-\delta}|u|_{2}\big) \\
\le & C\big(1+|u_t|_{2}+|u_x|_{2}^{\frac32}\big)+\frac12|u_{xx}|_{2}\label{e-1.25},
\end{align*}
which implies that  
\begin{equation}\label{e-1.290}
|u_{xx}|_{2}\leq C\Big(1+|u_t|_{2}+|u_x|_{2}^{\frac32}\Big).
\end{equation}

Consequently, according to \eqref{j4}-\eqref{e-1.290}, Lemma \ref{l4.3} and Young's inequality, one can obtain 
\begin{equation}\label{j14}
\text{J}_4\leq C\Big(1+|\rho^{\frac{\delta-1}{2}}u_x|_{2}^{\frac{10}{3}}\Big)+\frac18|u_t|_{2}^2.
\end{equation}

For the terms $\text{J}_i \ (i=5,6,7)$, one can similarly obtain 
\begin{equation}\label{strange1}
    \begin{split}
       \text{J}_5=&-\int u u_x u_t{\rm{d}}x\leq C|u|_{\infty}|u_x|_{2}|u_t|_{2}\\
\leq& C|u|_{2}^{\frac12}|u_x|_{2}^{\frac12}|\rho|_{\infty}^{\frac{1-\delta}{2}}|\rho^{\frac{\delta-1}{2}}u_x|_{2}|u_t|_{2}\\
\leq& C|\rho^{\frac{\delta-1}{2}}u_x|_{2}^3+\frac18 |u_t|_{2}^2,\\
\text{J}_6=&-A\gamma\int \rho^{\gamma-2}\rho_x u_t{\rm{d}}x
=  -\frac{A\gamma}{\alpha}\int \rho^{\gamma-\delta} (v-u) u_t{\rm{d}}x\\
\le & C \Big(|\rho^{\gamma-\delta}|_{2} |v|_\infty + |\rho|_{\infty}^{\gamma-\delta} |u|_2\Big) |u_t|_2
\le C+\frac18 |u_t|^2_2,\\
\text{J}_7=&\alpha \int \rho^{\delta-2}\rho_x u_x u_t{\rm{d}}x=\int (v-u)u_x u_t{\rm{d}}x\\
\leq &C\left(|v|_{\infty}+|u|_{\infty}\right)|u_x|_{2}|u_t|_{2}\\
\leq &C(1+|u_x|_{2}) |\rho^{\frac{\delta-1}{2}}u_x|_2^2+\frac18|u_t|_{2}^2.
    \end{split}
\end{equation}

Substituting \eqref{j4}-\eqref{strange1} into \eqref{e-1.20}, one gets
\begin{equation}\label{e-1.23}
\frac{\text{d}}{\text{d}t}|\rho^{\frac{\delta-1}{2}}u_x|_2^2+|u_t|_2^2\leq C\Big(1+|\rho^{\frac{\delta-1}{2}}u_x|_{2}^{\frac{10}{3}}\Big),
\end{equation}
which, along with  the Gronwall inequality and \eqref{e-1.290}, yields that 
 \begin{equation}\label{eg}
 |\rho^{\frac{\delta-1}{2}}u_x(t,\cdot)|_2^2+\int_0^t\big( |u_t|_2^2+|\rho^{\delta-1}u_{xx}|_2^2\big)(s,\cdot){\rm{d}}s\leq C \quad {\rm{for}}\ \ 0\le t\le T.
\end{equation}
\par\medskip

\textbf{Case 2:} $\delta=1$. First, similarly to the derivation of \eqref{e-1.290}, one still has 
\begin{equation}\label{e-1.290dengyuyi}
|u_{xx}|_{2}\leq C\Big(1+|u_t|_{2}+|u_x|_{2}^{\frac32}\Big).
\end{equation}

Second, multiplying $\eqref{e2.1}_2$ by $u_t$ and integrating  over $\mathbb R$, one has
\begin{equation*}
\begin{aligned}
&\frac{\alpha}{2}\frac{\text{d}}{\text{d}t}| u_x|_2^2+| u_t|_2^2\\
=& -\int u u_x u_t{\rm{d}}x-A\gamma\int \rho^{\gamma-2}\rho_x u_t{\rm{d}}x+\alpha \int \rho^{-1}\rho_x u_x u_t{\rm{d}}x.
\end{aligned}
\end{equation*}
Via the similar arguments for estimating $\text{J}_5$-$\text{J}_7$, one can get
 \begin{equation}\label{e4.29}
| u_x(t,\cdot)|_2^2 +\int_0^t( |u_t(s,\cdot)|_2^2+|u_{xx}(s,\cdot)|_2^2){\rm{d}}s\leq C \quad {\rm{for}}\ \ 0\leq t\leq T.
\end{equation}

Finally, it follows from \eqref{eg}, \eqref{e4.29} and    Sobolev embedding theorem that
\begin{equation*}\label{e4.31}
\begin{split}
\int_0^t |u_{x}(s,\cdot)|_{\infty}{\rm{d}}s\leq C\int_0^t |u_x(s,\cdot)|_{2}^{\frac12}|u_{xx}(s,\cdot)|_{2}^{\frac12}{\rm{d}}s
\leq C \quad {\rm{for}}&\ \ 0\leq t\leq T.
\end{split}
\end{equation*}

The proof of Lemma \ref{l4.6} is complete. 
\end{proof}

\subsection{The  second  order  estimates of $u$}
 
We consider the second order estimates of $u$. For this purpose,  we first  give the $L^2$ and $L^4$ estimates for $\psi$.
 
 \begin{lem}\label{l4.7}
Assume \eqref{zuizhong} additionally. Then for any $T>0$, it holds that 
 \begin{equation*}\label{e4.34}
|\psi(t,\cdot)|_{2}+|\psi(t,\cdot)|_{4}\leq C \quad {\rm{for}}\ \  0\leq t\leq T.
\end{equation*}
 \end{lem}
 
 \begin{proof}
  Multiplying \eqref{e-1.16} by $v$ and integrating the resulting equation over $\mathbb R$, one has
 \begin{equation*}\begin{split}
\frac{\text{d}}{\text{d}t}|v|_2^2\leq  C\left(|u_x|_{\infty}|v|_{2}^2+|\rho|^{\gamma-\delta}_{\infty}|v|_{2}^2+|\rho|^{\gamma-\delta}_{\infty}|u|_{2}|v|_{2}\right)
\le   C(1+|u_x|_{\infty} )|v|_{2}^2+C,
\end{split}
 \end{equation*}
which, along with Lemma \ref{l4.6} and the Gronwall inequality,  yields that
\begin{equation}\label{e4.35}
|v(t,\cdot)|_2\leq C \quad {\rm{for}}\ \ 0\leq t\leq T.
\end{equation}
It thus follows from \eqref{e4.35}, Lemmas \ref{l4.4}-\ref{l4.6} and the Sobolev embedding theorem  that 
\begin{equation*}
\begin{split}
|\psi(t,\cdot)|_{2}\leq C\left(|v(t,\cdot)|_{2}+|u(t,\cdot)|_{2}\right)\leq C\quad {\rm{for}}&\ \ 0\leq t\leq T,\\
|\psi(t,\cdot)|_{4}\leq C\left(|v(t,\cdot)|_{4}+|u(t,\cdot)|_{4}\right)\leq C\quad {\rm{for}}&\ \ 0\leq t\leq T.
\end{split}
\end{equation*}

The proof of Lemma \ref{l4.7} is complete. 
 \end{proof}

Now we are ready to show the second order estimates of $u$.
\begin{lem}\label{l4.8}
Assume \eqref{zuizhong} additionally. Then for any $T>0$,  it holds that 
 \begin{equation*}
|u_t(t,\cdot)|^2_2+|u_{xx}(t,\cdot)|_2^2+|\rho^{\delta-1} u_{xx}(t,\cdot)|_2^2+\int_0^t |\rho^{\frac{\delta-1}{2}}u_{tx}(s,\cdot)|^2_2 {\rm{d}}s  \le C \quad {\rm{for}}\ \ 0\leq t\leq T.
\end{equation*}

\end{lem}
 
\begin{proof} 
\textbf{Case 1:} $0<\delta<1$.
 First,  according to  Lemmas \ref{l4.3},  \ref{l4.5}-\ref{l4.7} and the Sobolev embedding theorem, one has
\begin{equation*}
\begin{split}
&|\phi_x|_4=|A\gamma\rho^{\gamma-\delta}\rho^{\delta-2}\rho_x|_4\\
\leq & C|\rho|^{\gamma-\delta}_{\infty} |\rho^{\delta-2}\rho_x|_4\leq C(|v|_4+|u|_4)\leq C.
\end{split}
\end{equation*}
Then it follows from $\eqref{e2.2}_1$ that 
\[
|\phi_t|_2\le C\left(|u|_4|\phi_x|_4+|\phi|_\infty|u_x|_2\right)\leq C.
\]
According to \eqref{e-1.290} and Lemma \ref{l4.6}, one has
\begin{equation}\label{uxxl1}
\begin{split}
|u_{xx}|_2\leq &  C\big(1+|u_t|_{2}+|u_x|_{2}^{\frac32}\big)\le C\big(1+|u_t|_2\big),\\
|\rho^{\delta-1}u_{xx}|_{2}\leq & C\big(1+|u_t|_{2}+|u_x|_{2}^{\frac32}\big)+\frac12|u_{xx}|_{2}\le C\big(1+|u_t|_2\big).
\end{split}
\end{equation}

Second, differentiating $(\ref{e2.2})_2$ with respect to $t$, 
multiplying the resulting equation by $u_t$ and integrating over $\mathbb{R}$, one has 
\begin{equation}\label{util2}
\begin{split}
&\frac{1}{2} \frac{\text{d}}{\text{d}t}|u_t|^2_2+\alpha|\rho^{\frac{\delta-1}{2}}u_{tx}|^2_2\\
=&\int \Big(\alpha\big(\rho^{\delta-1}\big)_t u_{xx}-\alpha \big(\rho^{\delta-1}\big)_x u_{tx}+\alpha (\psi u_x)_t -(u u_x)_t - \phi_{tx} \Big)u_t {\rm{d}}x \triangleq  \sum_{i=8}^{12}\text{J}_i.
\end{split}
\end{equation}

According to the equations $\eqref{e1.1}_1$ and $\eqref{e2.2}_3$, H\"older's inequality, Young's inequality  and Lemmas \ref{l4.3} and  \ref{l4.5}-\ref{l4.7}, one gets
  \begin{equation}\label{j18}
\begin{split}
\text{J}_8=&\alpha\int  \big(\rho^{\delta-1}\big)_t u_{xx} u_t{\rm{d}}x 
=\alpha(\delta-1)\int  \rho^{\delta-2}\rho_t u_{xx} u_t{\rm{d}}x \\
= &-\alpha(\delta-1)\int \Big( \rho^{\delta-2}\rho_x u +\rho^{\delta-1}  u_x\Big)u_{xx} u_t{\rm{d}}x \\
\le &C\Big( |\rho^{\delta-2}\rho_x|_4 |u|_{4} |u_{xx}|_2  + \big|\rho^{\frac{\delta-1}{2}} u_x\big|_2 \big|\rho^{\frac{\delta-1}{2}} u_{xx}\big|_2\Big)|u_t|_{\infty}\\
\le &C\Big( |u_{xx}|_2  + \big|\rho^{\frac{\delta-1}{2}} u_{xx}\big|_2\Big)|u_t|_2^{\frac12} |u_{tx}|_2^{\frac12}\\
\le& C\Big( |u_t|_2^2 + | \rho^{\frac{\delta-1}{2}} u_{xx}|_2^2\Big)+\frac{\alpha}{8}|\rho^{\frac{\delta-1}{2}}u_{tx}|^2_2,\\
\text{J}_9=&-\alpha\int \big(\rho^{\delta-1}\big)_x u_{tx}u_t{\rm{d}}x=-\alpha(\delta-1)\int \rho^{\delta-2} \rho_x u_{tx}u_t {\rm{d}}x\\
\leq &C| \rho^{\delta-2} \rho_x|_4 |u_{tx}|_2|u_t|_4\\
\le & C|u_{tx}|_2^{\frac54}|u_t|_2^{\frac34}
\leq C|u_t|_2^2+ \frac{\alpha}{8}|\rho^{\frac{\delta-1}{2}}u_{tx}|^2_2,
\end{split}
\end{equation}
where one has used the facts
$$
|u_t|_\infty\leq C|u_{t}|_2^{\frac12}|u_{tx}|_2^{\frac12}\quad \text{and} \quad |u_t|_4\leq C|u_{tx}|_2^{\frac14}|u_t|_2^{\frac34}.
$$
Similarly, via the integration by parts, one has 
 \begin{equation}\label{j20}
\begin{split}
\text{J}_{10}=&\alpha\int (\psi u_x)_t u_t {\rm{d}}x \\
=&\alpha\int \Big(\psi u_{tx} u_t-\big((u \psi)_x+(\delta-1)\psi u_x+\delta\rho^{\delta-1} u_{xx}\big)u_x u_t\Big){\rm{d}}x\\
\leq &C\Big(|\psi|_4\big(|u_{tx}|_2|u_t|_{4}+|u|_\infty|u_{xx}|_2|u_t|_{4}+ |u|_4|u_x|_{\infty} |u_{tx}|_2\\
&+|u_x|_2|u_t|_4 |u_x|_{\infty}\big)+|\rho^{\delta-1}u_{xx}|_2|u_x|_2|u_t|_\infty\Big)\\
\leq &C\Big(1+ |u_t|_2^2 + | \rho^{\delta-1}u_{xx}|_2^2\Big)+\frac{\alpha}{8}|\rho^{\frac{\delta-1}{2}}u_{tx}|^2_2,\\
\text{J}_{11}=&-\int (u u_x)_t u_t {\rm{d}}x 
=- \int( u u_{tx} u_t+u_t^2 u_x ){\rm{d}}x\\
\le & C\big( |u|_{\infty} |u_{tx}|_2 |u_t|_2+ |u_x|_{\infty} |u_t|_2^2\big)\\
\leq & C\Big(1+|u_{xx}|_2^{\frac12}\Big) |u_t|_2^2+\frac{\alpha}{8}|\rho^{\frac{\delta-1}{2}}u_{tx}|^2_2,\\
\text{J}_{12}=& -\int \phi_{tx}  u_t{\rm{d}}x = \int \phi_t  u_{tx}{\rm{d}}x \\
\le& C |\phi_t|_2 |u_{tx}|_2\le C+\frac{\alpha}{8}|\rho^{\frac{\delta-1}{2}}u_{tx}|^2_2.
\end{split}
\end{equation}

Substituting \eqref{j18}-\eqref{j20} into \eqref{util2}, one can obtain 
\begin{equation}\label{lg}
 \frac{\text{d}}{\text{d}t}|u_t|^2_2+|\rho^{\frac{\delta-1}{2}}u_{tx}|^2_2\le C \Big(1+\big(1+|u_{xx}|_2^{\frac12} \big)| u_t|^2_2+ | \rho^{\delta-1} u_{xx}|_2^2\Big).
\end{equation}
Integrating \eqref{lg} over $(\tau, t)(\tau\in (0,t))$, one has 
\begin{equation}\label{lg1}
\begin{split}
&|u_t(t,\cdot)|_2^2+\int_\tau^t |\rho^{\frac{\delta-1}{2}}u_{tx}(s,\cdot)|_2^2 {\rm{d}}s\\
\leq& |u_t(\tau,\cdot)|_2^2+C\int_\tau^t\Big(1+\big(1+|u_{xx}|_2^{\frac12} \big)| u_t|^2_2+ | \rho^{\delta-1} u_{xx}|_2^2\Big)(s,\cdot) {\rm{d}}s.
\end{split}
\end{equation}
It follows from $\eqref{e2.2}_2$ that 
\begin{equation}
|u_t(\tau,\cdot)|_2\leq C\Big( |u|_\infty|u_x|_2+|\phi_x|_2+|\rho^{\delta-1}u_{xx}|_2+|\psi|_2|u_x|_\infty\Big)(\tau),
\end{equation}
which, together with the time continuity of $(\rho,u)$ and \eqref{e1.7}-\eqref{e1.8},  implies that
\begin{equation*}
\lim\sup_{\tau\to 0} |u_t(\tau,\cdot)|_2\leq C\big(|u_0|_\infty|\partial_x u_0|_2+|\partial_x\phi_0|_2+|g_2|_2+|\psi_0|_2|\partial_x u_0|_\infty\big)\leq C. 
\end{equation*}
Letting $\tau\to 0$ in \eqref{lg1} and using the Gronwall inequality, one can obtain
\[
|u_t(t,\cdot)|^2_2+\int_0^t |\rho^{\frac{\delta-1}{2}}u_{tx}(s,\cdot)|^2_2 \text{d}s \le C \quad {\rm{for}}\ \ 0\leq t\leq T.
\]
It thus follows from  \eqref{uxxl1} that
\[
|u_{xx}(t,\cdot)|_2+ |\rho^{\delta-1}u_{xx}(t,\cdot)|_2 \le C\big(1+|u_t(t,\cdot)|_2\big)\le C  \quad {\rm{for}}\ \ 0\leq t\leq T.
\]

\textbf{Case 2:} $\delta=1$. 
First,  according to  $(\ref{e2.1})_2$,  H\"older's inequality, Young's  inequality and  Lemmas \ref{l4.5}-\ref{l4.7},  one has
\begin{equation*}
\begin{split}
|u_{xx}|_2=&\frac{1}{\alpha}\Big| u_t+u u_x+\phi_x-\alpha\psi u_x\Big|_2\\
\le & C \big(|u_t|_2+|u|_{\infty} |u_x|_2+ |\phi_x|_2+|\psi |_4 |u_x|_{4}\big)\\
\le & C \big(1+|u_t|_2+ |u_x|_{2}^{\frac34} |u_{xx}|_2^{\frac14}\big)\\
\le & C \big(1+|u_t|_2\big)+\frac12|u_{xx}|_2,
\end{split}
\end{equation*}
which  implies that 
\begin{equation}\label{uxx}
\begin{split}
|u_{xx}|_{2}\le C(1+|u_t|_2).
\end{split}
\end{equation}

Second, differentiating $(\ref{e2.1})_2$ with respect to $t$, 
multiplying the resulting equation by $u_t$ and integrating over $\mathbb{R}$,  one has 
\begin{equation}\label{zhu38}
\frac{1}{2} \frac{\text{d}}{\text{d}t}|u_t|^2_2+\alpha|u_{tx}|^2_2
=-\int \Big(u u_x+ \phi_x - \alpha\psi u_x\Big)_t u_t \text{d}x\triangleq  \sum_{i=13}^{15}\text{J}_i.
\end{equation}

Then according to   H\"older's inequality, Young's  inequality, the Sobolev embedding theorem  and Lemmas \ref{l4.1}-\ref{l4.7}, one has 
\begin{equation*}
\begin{split}
\text{J}_{13}=&-\int  (u u_x)_t u_t \text{d}x
=-\int  \big( |u_t|^2 u_x  +u u_{tx} u_t\big)\text{d}x\\ 
\le & C\big(|u_x|_\infty| u_t|^2_2+|u|_{\infty} |u_{tx}|_2 |u_t|_2\big)\\
\le& C\big(1+|u_x|_\infty)| u_t|^2_2+\frac{\alpha}{8} |u_{tx}|_2^2,\\ 
\text{J}_{14}=&- \int  \phi_{tx}  u_t \text{d}x
=\int   \phi_t  u_{tx} \text{d}x\\
\le &C|\phi_t|_2|u_{tx}|_2\leq C+\frac{\alpha}{8} |u_{tx}|_2^2,\\
\text{J}_{15}=&\alpha\int (\psi u_x)_t u_t \text{d}x
=\alpha\int \Big( \psi u_{tx} u_t+\psi_t u_x u_t\Big)\text{d}x\\ 
=&\alpha\int \Big( \psi u_{tx} u_t-(u\psi)_x u_x u_t-u_{xx}u_x u_t\Big)\text{d}x\\ 
\le & C\Big(|\psi|_2  |u_{tx}|_2|u_t|_{\infty}+|u|_{\infty}|\psi|_2\big( |u_{xx}|_2|u_t|_{\infty}+|u_x|_{\infty}|u_{tx}|_2\big)\\
&+|u_x|_\infty |u_{xx}|_2| u_t|_2\Big)\\
\le & C\Big(  |u_t|_2^{\frac12}|u_{tx}|_2^{\frac32}+ |u_{xx}|_2|u_t|_2^{\frac12}|u_{tx}|_2^{\frac12}+|u_{xx}|_2^{\frac12}|u_{tx}|_2+ |u_{xx}|_2^{\frac32}| u_t|_2\Big)\\
\le &C\big(1+|u_{xx}|_2^2\big)\big(1+| u_t|^2_2\big)+\frac{\alpha}{8} |u_{tx}|_2^2,
\end{split}
\end{equation*}
which, along with  \eqref{uxx}-(\ref{zhu38}), yields that 
\begin{equation}\label{uxt}
\begin{split}
& \frac{\text{d}}{\text{d}t}|u_t|^2_2+|u_{tx}|^2_2\le C \big(1 +|u_{xx}|_2^2\big)\big(1+|u_t|^2_2\big)\leq  C \big(1+|u_t|^2_2\big)^2.
\end{split}
\end{equation}
Integrating (\ref{uxt}) over $(\tau,t)$ $(\tau \in( 0,t))$, one has
\begin{equation}\label{ut0}
\begin{split}
&|u_t(t,\cdot)|^2_2+\int_\tau^t| u_{tx}(s,\cdot)|^2_2\text{d}s
\le |u_t(\tau,\cdot)|^2_2+C\int_{\tau}^t  \big(1+|u_t(s,\cdot)|^2_2\big)^2\text{d}s.
\end{split}
\end{equation}
According to $(\ref{e2.1})_2$, one can obtain
\begin{equation*}
\begin{split}
|u_t(\tau,\cdot)|_2
\le & C\big(|u u_x|_2+|\phi_x|_2+|u_{xx}|_2+|\psi u_x|_2\big)(\tau)\\
\le & C\big( |u|_\infty | u_x|_2+| \phi_x|_2+|u_{xx}|_{2}+|\psi|_2|u_x|_\infty\big)(\tau),
\end{split}
\end{equation*}
which, along with the time continuity of $(\rho,u)$ and \eqref{id1}, yields that 
\begin{equation*}
\begin{split}
&\lim \sup_{\tau\rightarrow 0}|u_t(\tau,\cdot)|_2
\le C\big( |u_0|_\infty \| u_0\|_1+\| \phi_0\|_1+\|u_0\|_2+|\psi_0|_2 \|u_0\|_2\big)\le C.
\end{split}
\end{equation*}
Letting $\tau \rightarrow 0$ in (\ref{ut0}),  it follows from the Gronwall  inequality that
$$
|u_t(t,\cdot)|^2_2+|u_{xx}(t,\cdot)|_2^2+\int_0^t| u_{tx}(s,\cdot)|^2_2\text{d}s\leq C \quad {\rm{for}}\ \ 0\leq t\leq T.
$$

The proof of Lemma \ref{l4.8} is complete.
\end{proof}

\subsection{The  high  order  estimates of $\phi$ and $ \psi$}

The following lemma provides  the high order estimates of $\phi$ and $ \psi$.

\begin{lem}\label{l4.9}  
Assume \eqref{zuizhong} additionally. Then for any $T>0$,  it holds that 
\begin{equation*}
\begin{split}
&|\phi_{xx}(t,\cdot)|^2_2+|\phi_{tx}(t,\cdot)|^2_2+|\psi_x(t,\cdot)|_2^2+|\psi_t(t,\cdot)|^2_{2}\\
&+\int_{0}^{t}\big(|u_{xxx}|^2_2+|\rho^{\frac{\delta-1}{2}}u_{xxx}|_2^2+|\rho^{\delta-1}u_{xxx}|_2^2+|\phi_{tt}|^2_{2}\big)(s,\cdot){\rm{d}}s\leq C\quad {\rm{for}}\ \ 0\leq t\leq T.
\end{split}
\end{equation*}

 \end{lem}

 \begin{proof} 
 \textbf{Case 1:} $0<\delta<1.$ 
 The proof will be divided into two steps.
 
\textbf{Step 1:} the estimate on $u_{xxx}$.
First, it follows from $\eqref{e2.2}_2$,  H\"older's inequality, Young's inequality  and Lemmas \ref{l4.1}-\ref{l4.8} that 
 \begin{equation}\label{uxxxl1}
\begin{split}
|u_{xxx}|_2 = & |\rho^{1-\delta}\rho^{\delta-1}u_{xxx}|_2 \le C|\rho|_{\infty}^{1-\delta}|\rho^{\delta-1}u_{xxx}|_2\le C|\rho^{\delta-1}u_{xxx}|_2 \\
\le& C\big(|u_{tx}|_{2}+|(u  u_x)_x|_2+|\phi_{xx}|_2+|(\psi u_x)_x|_2+|(\rho^{\delta-1})_x u_{xx}|_2\big)\\
\le & C\big(|u_{tx}|_2+|u|_{\infty} | u_{xx}|_{2}+| u_x|_{\infty}| u_x|_2+|\phi_{xx}|_2+|\psi|_{\infty}|u_{xx}|_{2}+|\psi_x|_2 | u_x|_{\infty}\big)\\
\le & C\big(1+|u_{tx}|_2+| \psi_x|_2+|\phi_{xx}|_2\big).
\end{split}
\end{equation}

Second, we consider   the last two terms  on the right-hand side of \eqref{uxxxl1}. 
For $|\psi_x|_2$, differentiating $(\ref{e2.2})_3$ with respect to $x$, multiplying the resulting equation  by $\psi_x$ and integrating over $\mathbb{R}$, one can obtain 
 \begin{equation}\label{he1}
 \begin{split}
\frac{\text{d}}{\text{d}t}| \psi_x|^2_2
\le & C\int \big( |u_x \psi_x \psi_x| +  |u_{xx} \psi \psi_x| +| \rho^{\delta-1}u_{xxx} \psi_x| \big)\text{d} x \\
\le& C\big(  1+ |\psi_x|_2^2 +|\rho^{\delta-1} u_{xxx}|_2^2   \big).
\end{split}
\end{equation}

For $|\phi_{xx}|_2$. Applying $\partial_{x}^2$ to $\eqref{e2.2}_1$,  multiplying the resulting equation by $\phi_{xx}$  and integrating over $\mathbb{R}$, one has
\begin{equation}\label{phixxl1}
\begin{split}
\frac{\text{d}}{\text{d}t}| \phi_{xx}|^2_2
\le& C\int  \big(|u_{xx} \phi_x | + |u_x \phi_{xx}|+|\phi u_{xxx}|\big) |\phi_{xx}|{\rm{d}}x \\
\le &C  \big(|u_{xx} |_2 |\phi_x |_{\infty} + |u_x |_{\infty} |\phi_{xx}|_2+|\phi|_{\infty}| u_{xxx}|_2\big) |\phi_{xx}|_2\\
\le &C  \big( 1+|\phi_{xx}|_2^2+| u_{xxx}|_2^2\big). 
\end{split}
\end{equation}

According to \eqref{uxxxl1}-\eqref{phixxl1}, one can obtain 
 \begin{equation*}
\frac{\text{d}}{\text{d}t}\big(| \psi_x|^2_2+| \phi_{xx}|^2_2\big)
\le C\big(  1+|u_{tx}|_2^2+ |\psi_x|_2^2 +| \phi_{xx}|_2^2    \big),
\end{equation*}
which, along with the Gronwall  inequality and Lemma \ref{l4.8}, yields that
\[
|\psi_x(t,\cdot)|_2+|\phi_{xx}(t,\cdot)|_2\le C \quad {\rm{for}}\ \ 0\leq t\leq T.
\]
At last, it  thus follows from \eqref{uxxxl1} and Lemma \ref{l4.8} that
\begin{equation}
\begin{split}
&\int_{0}^{t}\big(|u_{xxx}|^2_2 + |\rho^{\delta-1}u_{xxx}|_2^2+|\rho^{\frac{\delta-1}{2}}u_{xxx}|_2^2\big)(s,\cdot) \text{d} s\\
 \le &C \int_{0}^{t}\big(  1+ |\psi_x|_2^2+|u_{tx}|_2^2+| \phi_{xx}|_2^2   \big) (s,\cdot)\text{d} s\le C\quad {\rm{for}}\ \ 0\leq t\leq T.
\end{split}
\end{equation}

\textbf{Step 2:} estimates on time derivatives of the density. First, according to   $\eqref{e2.2}_3$,  H\"older's  inequality and Lemma \ref{l4.8}, one has
\begin{equation*}
\begin{split}
|\psi_t|_2\le& C\big(|\rho^{\delta-1}u_{xx}|_2+| u \psi_x|_2+|u_x \psi|_2\big)\\
\le& C\big(1+|u|_{\infty}|\psi_x|_2+| u_x|_{\infty}|\psi|_2\big)\le C.
\end{split}
\end{equation*}

Second, according to  $\eqref{e2.2}_1$  and  H\"older's  inequality, one has
\begin{equation*}
\begin{split}
|\phi_{tx}|_2 \le & C\big(|u\phi_{xx}|_2+|u_x \phi_x|_2+|\phi u_{xx}|_2\big)\\
 \le& C\big(| u|_{\infty}  |\phi_{xx}|_2+| u_x|_{\infty} |\phi_x|_2+|\phi|_{\infty} |u_{xx}|_2\big)\le C,\\
|\phi_{tt}|_2\le& C(|u_t\phi_x|_2+|u \phi_{tx}|_2+|\phi_t u_x|_2+|\phi u_{tx}|_2)\\
\le & C(|u_t|_2|\phi_x|_{\infty}+|u|_{\infty}|\phi_{tx}|_2+|\phi_t|_2| u_x|_{\infty}+|\phi|_{\infty}| u_{tx}|_2)\\
\le & C\big(1+|u_{tx}|_2\big), 
\end{split}
\end{equation*}
which, along with Lemma \ref{l4.8}, yields that
\[
 \int_{0}^{t}|\phi_{tt}(s,\cdot)|^2_{2}\text{d}s \leq C\quad {\rm{for}}\ \ 0\leq t\le T.
\]

 \textbf{Case 2:} $\delta=1$. 
  The proof will be  divided into two steps.\par

\textbf{Step 1:} the estimate on $u_{xxx}$.
First, it follows  from $\eqref{e2.1}_2$, H\"older's inequality and Lemmas \ref{l4.1}-\ref{l4.8} that 
 \begin{equation}\label{uxxx}
\begin{split}
|u_{xxx}|_{2} 
\le& C\big(|u_{tx}|_{2}+|(u  u_x)_x|_2+|\phi_{xx}|_2+|(\psi u_x)_x|_2\big)\\
\le & C\big(|u_{tx}|_2+|u|_{\infty} | u_{xx}|_{2}+| u_x|_{\infty}| u_x|_2+|\phi_{xx}|_2\\
&+|\psi|_{\infty}|u_{xx}|_{2}+|\psi_x|_2 | u_x|_{\infty}\big)\\
\le & C\big(1+|u_{tx}|_2+| \psi_x|_2+| \phi_{xx}|_{2} \big).
\end{split}
\end{equation}

Second, for $|\psi_x|_2$, differentiating $(\ref{e2.1})_3$ with respect to $x$,  multiplying the resulting equation by $\psi_x$ and using the integration by part, one can obtain
 \begin{equation}
 \begin{split}
\frac{\text{d}}{\text{d}t}| \psi_x|^2_2
\le & C\int \big( |u_x \psi_x \psi_x| +  |u_{xx} \psi \psi_x| +| u_{xxx} \psi_x| \big)\text{d} x\\
\le& C\big(  |u_x |_{\infty} |\psi_x|_2^2+ |u_{xx}|_2 |\psi|_{\infty} |\psi_x|_2 +| u_{xxx}|_2 |\psi_x|_2   \big)\\
\le& C\big(  1+ |\psi_x|_2^2 +| u_{xxx}|_2^2   \big).
\end{split}
\end{equation}
For $|\phi_{xx}|_2$,  by the definition of $\phi$ and $\psi$, one has
\begin{equation}\label{phixx}
\begin{split}
|\phi_{xx}|_2=&(\gamma-1)| (\psi \phi)_x|_2\\
\le & C\big( |\psi_x|_2 |\phi|_{\infty} +|\psi|_{\infty} |\phi_x|_2\big)\le C(1+ |\psi_x|_2).
\end{split}
\end{equation}
Thus, combining \eqref{uxxx}-\eqref{phixx} together, one  deduces that
 \begin{equation}\label{psix}
\frac{\text{d}}{\text{d}t}| \psi_x|^2_2
\le C\big(  1+ |\psi_x|_2^2 +| u_{tx}|_2^2   \big),
\end{equation}
which, together with the Gronwall inequality and Lemma \ref{l4.8}, yields that
 \begin{equation}\label{pnn}
|\psi_x(t,\cdot)|_2\le C \quad {\rm{for}}\ \ 0\leq t\leq T.
\end{equation}
It thus follows from \eqref{uxxx},  \eqref{phixx} and \eqref{pnn} that 
\begin{equation*}
\begin{split}
|\phi_{xx}(t,\cdot)|_2\le C(1+|\psi_x(t,\cdot)|_2)\le C \quad {\rm{for}}& \ \ 0\leq t\leq T,\\
\int_{0}^{t}|u_{xxx}(s,\cdot)|^2_2 \text{d} s \le C \int_{0}^{t}\big(  1 +| u_{tx}(s,\cdot)|_2^2   \big) \text{d} s\le C \quad {\rm{for}}& \ \ 0\leq t\leq T.
\end{split}
\end{equation*}

\textbf{Step 2:} estimates on time derivatives of the density.
The boundedness of $|\psi_t(t,\cdot)|_2$, $|\phi_{tx}(t,\cdot)|_2$ and  $\int_0^t |\phi_{tt}(s,\cdot)|_2^2 {\rm{d}}s $  can be shown via similar argument as in the case $0<\delta<1$. 

The proof of Lemma \ref{l4.9} is complete.
\end{proof}

\subsection{ Time-weighted energy estimates of  $u$}
At last,  we establish the time-weighted energy estimates of  $u$.

\begin{lem}\label{l4.10}  
Assume \eqref{zuizhong} additionally. Then for any $T>0$,  it holds that 
\begin{equation*}
\begin{split}
&t|\rho^{\frac{\delta-1}{2}} u_{tx}(t,\cdot)|^2_2+\int_{0}^{t} s \big(|u_{tt}|^2_2+ |u_{txx}|_2^2 \big)(s,\cdot){\rm{d}}s\leq C\quad {\rm{for}}\ \ 0\leq t\leq T.
\end{split}
\end{equation*}
\end{lem}
\begin{proof}
\textbf{Case 1:} $0<\delta<1$.
First, differentiating $(\ref{e2.2})_2$ with respect to $t$, multiplying the resulting equation  by $ u_{tt}$ and integrating over $\mathbb{R}$, one arrives at
\begin{equation}\label{etrq}
\begin{split}
&\frac{\alpha}{2}\frac{\text{d}}{\text{d}t} |\rho^{\frac{\delta-1}{2}}u_{tx}|_2^2+ |u_{tt}|_2^2\\
=&\int \Big(\alpha(\rho^{\delta-1})_t u_{xx}-\alpha(\rho^{\delta-1})_x u_{tx}-(u u_x)_t\\
&-\phi_{tx}+\alpha(\psi u_x)_t\Big)u_{tt}\text{d}x+\frac{\alpha}{2}\int (\rho^{\delta-1})_t u_{tx}^2\text{d}x \triangleq \sum_{i=16}^{21} \text{J}_{i}.
\end{split}
\end{equation}

According to the equation  $\eqref{e1.1}_1$, H\"older's inequality, Lemma \ref{al1} and Lemmas \ref{l4.1}-\ref{l4.9}, one can obtain that  
\begin{equation*}
\begin{split}
\text{J}_{16}=&\alpha\int (\rho^{\delta-1})_t u_{xx} u_{tt}\text{d}x\\
=& -\alpha\int\big((\rho^{\delta-1})_x u+(\delta-1)\rho^{\delta-1}u_x\big)u_{xx}u_{tt}\text{d}x\\
\leq &C\big(|\psi|_\infty|u|_\infty|u_{xx}|_2+|\rho^{\delta-1}u_{xx}|_2|u_x|_\infty\big)|u_{tt}|_2\leq C+\frac18 |u_{tt}|_2^2,\\
\sum_{i=17}^{20}\text{J}_{i}=&\int \Big(-\alpha(\rho^{\delta-1})_x u_{tx}-(u u_x)_t-\phi_{tx}+\alpha(\psi u_x)_t\Big)u_{tt}\text{d}x\\
\leq &C\big(|u_x|_2|u_t|_\infty+|u|_\infty|u_{tx}|_2+|\phi_{tx}|_2+|\psi|_\infty|u_{tx}|_2+|\psi_t|_2|u_x|_\infty\big)|u_{tt}|_2\\
\leq &C(1+|\rho^{\frac{\delta-1}{2}}u_{tx}|_2^2)+\frac18|u_{tt}|_2^2,\\
\text{J}_{21}=&\frac{\alpha}{2}\int (\rho^{\delta-1})_t u_{tx}^2\text{d}x\\
\leq & C\big(|\psi|_\infty|u|_\infty|u_{tx}|_2^2+|\rho^{\frac{\delta-1}{2}}u_{tx}|_2^2|u_x|_\infty\big)
\leq C|\rho^{\frac{\delta-1}{2}}u_{tx}|_2^2,
\end{split}
\end{equation*}
which, along with (\ref{etrq}), implies that 
\begin{equation}\label{etrq1}
\frac{\text{d}}{\text{d}t}|\rho^{\frac{\delta-1}{2}} u_{tx}|_2^2+ |u_{tt}|_2^2\leq C\big(1+|\rho^{\frac{\delta-1}{2}}u_{tx}|_2^2\big).
\end{equation}

Second, multiplying \eqref{etrq1} by $s$ and integrating with respect to $s$ over $[\tau, t]$ for $\tau\in(0,t)$, one has 
\begin{equation}\label{etrq2}
t|\rho^{\frac{\delta-1}{2}} u_{tx}(t,\cdot)|_2^2+ \int_\tau^t s|u_{tt}(s,\cdot)|_2^2\text{d}s\leq \tau|\rho^{\frac{\delta-1}{2}}u_{tx}(\tau,\cdot)|_2^2+C.
\end{equation}
Recalling the definition of regular solutions in Definition \ref{def2}, one has $$\rho^{\frac{\delta-1}{2}} u_{tx}\in L^2([0,T];L^2),$$
which, along with the Lemma \ref{bjr} (see in Appendix A), implies that there exists a sequence $s_k$ such that 
\begin{equation}
s_k\to 0, \quad \text{and}\quad  s_k|\rho^{\frac{\delta-1}{2}} u_{tx}(s_k, \cdot)|_2^2\to 0 \quad \text{as}\quad k\to \infty.
\end{equation}
After choosing $\tau=s_k\to 0$  in  \eqref{etrq2}, one can obtain
\begin{equation}\label{etrq3}
t|\rho^{\frac{\delta-1}{2}} u_{tx}(t,\cdot)|_2^2+ \int_0^t s|u_{tt}(s,\cdot)|_2^2\text{d}s\leq C \quad {\rm{for}}\ \ 0\leq t\leq T.
\end{equation}

Second, according to $(\ref{e2.2})_2$ and the estimates in Lemmas \ref{l4.1}-\ref{l4.9}, one has
 \begin{equation*}
 \begin{split}
 |u_{txx}|_2 
 =& |\rho^{1-\delta} \rho^{\delta-1} u_{txx}|_2\le  |\rho|_{\infty}^{1-\delta} |\rho^{\delta-1} u_{txx}|_2\\
 \le & C\big(| (\rho^{\delta-1})_t u_{xx}|_2+|u_{tt}|_{2}+|(u  u_x)_t|_2+|\phi_{tx}|_2+|(\psi u_x)_t|_2\big)\\
\le & C\big(1+|\rho^{\delta-1} u_x u_{xx}|_2 +|\rho^{\delta-2}\rho_x u u_{xx}|_2+ |u_{tt}|_2+|u_t|_2 |u_x|_{\infty}\\
&+|u|_{\infty} | u_{tx}|_{2}+|\psi|_{\infty}|u_{tx}|_{2}+|\psi_t|_2 | u_x|_{\infty}\big)\\
\le & C\big(1 + |u_{tt}|_2+ | u_{tx}|_{2}\big),
 \end{split}
\end{equation*}
which, along with \eqref{etrq3}, implies that
\[
\int_{0}^{t} s  |u_{txx}(s,\cdot)|_2^2{\rm{d}}s\leq C\quad {\rm{for}}\ \ 0\leq t\leq T.
\]

\textbf{Case 2:} $\delta=1$.
Differentiating $(\ref{e2.1})_2$ with respect to $t$, multiplying the resulting equation  by $ u_{tt}$ and integrating over $\mathbb{R}$, one has

\begin{equation*}\label{etrqq}
\frac{\alpha}{2}\frac{\text{d}}{\text{d}t}|u_{tx}|_2^2+ |u_{tt}|_2^2
=\int \Big(-(u u_x)_t-\phi_{tx}+\alpha(\psi u_x)_t\Big)u_{tt}\text{d}x.
\end{equation*}

Via the similar argument for estimating $\text{J}_{18}$-$\text{J}_{20}$, one can get

\begin{equation}\label{etrq33}
t|u_{tx}(t,\cdot)|_2^2+ \int_0^t s\big(|u_{tt}|_2^2+|u_{txx}|_2^2\big)(s,\cdot)\text{d}s\leq C \quad {\rm{for}}\ \ 0\leq t\leq T.
\end{equation}

The proof of Lemma \ref{l4.10} is complete.
\end{proof}

\subsection{Proof of Theorem \ref{th2} }\label{se46}

Based on the local-in-time well-posedness obtained in Lemmas \ref{l2}-\ref{vsi}  and the global-in-time a priori estimates established in Lemmas \ref{l4.1}-\ref{l4.10}, now  we are ready to give the  proof of Theorem \ref{th2}.

\textbf{Step 1}: the global  well-posedness of regular solutions.  
 First, under the initial assumption \eqref{e1.7}-\eqref{e1.8},  according to  Lemmas \ref{l2}-\ref{vsi}, there exists a unique regular solution $(\rho, u)$ to the Cauchy problem \eqref{e1.1}-\eqref{e1.3}  in $[0,T_*]\times \mathbb R$  for some  $T_*>0$, which  satisfies  \eqref{er1} with $T$ replaced by $T_*$.
Let $T^*>0$ be the life span  of the regular solution satisfying \eqref{er1}  obtained  in Lemmas \ref{l2}-\ref{vsi}.  It is obvious that $T^*\ge T_*$. Then we claim that $T^*=\infty$. Otherwise, if $T^*<\infty$, 
according to  the uniform a priori estimates obtained in Lemmas \ref{l4.1}-\ref{l4.10} and the standard weak compactness theory, one can know that for any sequence $\{t_k\}_{k=1}^\infty$ satisfying $0<t_k<T^*$ and 
$$t_k\to T^* \quad \text{as}\quad k\to \infty,$$ 
there exists a subsequence $\{t_{1k}\}_{k=1}^\infty$ and functions $(\phi, u, \psi)(T^*,x)$ such that 
\begin{equation}\label{f2}
\begin{aligned}
    &\phi(t_{1k},x)\rightharpoonup  \phi(T^*,x) \quad \text{in\,\,} H^2 \quad \text{as}\,\, k\to\infty,\\ 
     &u(t_{1k},x)\rightharpoonup  u(T^*,x) \quad \text{in\,\,} H^2 \quad  \text{as}\,\, k\to\infty,\\
     &\psi(t_{1k},x)\rightharpoonup  \psi(T^*,x) \quad \text{in\,\,} H^1 \quad  \text{as}\,\, k\to\infty.
\end{aligned}
\end{equation}

Second, we want  to show that functions $(\phi, u, \psi)(T^*,x)$ satisfy all the initial assumptions shown  in Lemma \ref{l2}, which include \eqref{e1.7}-\eqref{e1.8} except $\rho_0\in L^1$ and the following relationship between $\phi$ and $\psi$: 
\begin{equation}\label{ff} 
\psi=\frac{a\delta}{\delta-1}(\phi^{2e})_x.
\end{equation}
  It follows from \eqref{f2} that  \eqref{e1.7}  except $\rho_0\in L^1$ still   holds at the   time $t=T^*$.

 Next for  the relation in \eqref{ff}, we need to  consider the following equation
\begin{equation}\label{limit-3}
\phi_t+u\phi_x+(\gamma-1)\phi u_x=0,
\end{equation}
which holds in $[0,T^*)\times \mathbb{R}$ in the classical sense. Actually, if we regard  $(\phi(T^*,x), u(T^*,x), \\
\psi(T^*,x))$  and 
$$\phi_t(T^*,x)=-u(T^*, x)\phi_x(T^*, x)-(\gamma-1)\phi(T^*, x) u_x(T^*, x),$$ 
as the extended definitions of the regular solution $(\phi(t, x), u(t, x), \psi(t, x))$ at the  time $t=T^*$, then one has
$$
\text{ess}\sup_{0\leq t \leq T^*}(\|\phi(t,\cdot)\|_2+\|\phi_t(t,\cdot)\|_1)\leq C<\infty,$$
which, together with the Sobolev embedding theorem, implies that 
\begin{equation*}\label{f3}
    \phi(t, x)\in C([0,T^*];H^1).
\end{equation*}
It follows from the first line in  $\eqref{f2}$ and the consistency of weak convergence and  strong convergence in $H^1$ space that 
 \begin{equation}\label{f4}
    \phi(t_{1k},x)\to  \phi(T^*,x) \quad \text{in} \quad H^1 \quad \text{as}\quad  k\to\infty.
\end{equation} 
Notice that for any $0<R<\infty$, there exists a generic constant $C(T^*, R)$ such that 
\begin{equation}\label{f5}
\phi(t,x)\ge C(T^*, R)\quad \text{for\,\,any}\quad (t,x)\in[0,T^*]\times [-R, R],
\end{equation}
which, along with the last line in $\eqref{f2}$ and \eqref{f4}, implies that \eqref{ff} holds.

Moreover, it  follows from Lemmas \ref{l4.1}-\ref{l4.10} and \eqref{f2}-\eqref{ff} that 
\begin{equation}
   \sup_{\tau\leq t \leq T^*}( |\rho^{\frac{\delta-1}{2}} u_x(t,\cdot)|_2+|\rho^{\delta-1}u_{xx}(t,\cdot)|_2) \leq C, 
    \end{equation}
for any  $\tau\in (0, T^*]$, which implies that  $(\rho(T^*, x), u(T^*, x))$ satisfy all the   initial assumptions on the initial data of   Lemma \ref{l2}. 

Finally, if  we solve the system (\ref{e1.1}) with the initial time $T^*$, then Lemma \ref{l2} ensures that for some constant $T_0>0$, $(\rho, u)(t,x)$ is  the unique regular solution in $[T^*,T^*+T_0]\times \mathbb{R}$ to this problem, which satisfies  \eqref{er1} with $[0,T]$ replaced by $[T^*,T^*+T_0]$ except $\rho\in C([T^*,T^*+T_0];L^1)$. It follows from the boundedness of all required norms of the solution $(\rho,u)(t,x)$ in  $[0,T^*+T_0]\times \mathbb{R}$ and the standard arguments for proving the time continuity of the regular solution (see Section 3.6 (pages 138-139) of \cite{zz2} ) that, $(\rho, u)(t,x)$ is actually  the unique regular solution in $[0,T^*+T_0]\times \mathbb{R}$ to the Cauchy problem \eqref{e1.1}-\eqref{e1.3},  which satisfies  \eqref{er1} with $[0,T_*]$ replaced by $[0,T^*+T_0]$ except $\rho\in C([0,T^*+T_0];L^1)$.
Furthermore, according to  Lemma \ref{vsi}, one has  that the solution constructed above also satisfies  $\rho\in C([0,T^*+T_0];L^1)$.
Then the above conclusions contradict to the fact that $0<T^*<\infty$ is the maximal existence time of the unique regular solution $(\rho, u)(t,x)$ satisfying \eqref{er1}.

\begin{rk}
According to the  proof in this subsection, the definitions of $(\phi(T^*,x)$,\\
$u(T^*,x), \psi(T^*,x))$ do not depend on the choice of the time sequences $\{t_k\}_{k=1}^\infty$.
\end{rk}

\textbf{Step 2}: the global  well-posedness of classical solutions.  
Now we  show   that 
the regular solution  obtained  above  is indeed a classical one  in $(0, T]\times \mathbb{R}$ for any $T>0$.

First, due to   $\phi>0$ and 
$$\rho=\big(\frac{\gamma-1}{A\gamma}\phi\big)^{\frac{1}{\gamma-1}},$$
one has 
\begin{equation}\label{jichulianxu}
(\rho, \rho_t, \rho_x, u,  u_x)\in C([0,T]\times\mathbb R).
\end{equation}

Second, for the term $u_t$, according to   Lemma \ref{l4.10}, one has 
\begin{equation}\label{wtw}
\begin{split}
&t^{\frac12}u_{tx}\in L^\infty([0,T];L^2),\quad t^{\frac12}u_{tt}\in L^2([0,T];L^2),\\
& t^{\frac12}u_{txx}\in L^2([0,T];L^2), \quad t^{\frac12}u_{ttx}\in L^2([0,T];H^{-1}),
\end{split}
\end{equation}
which, along with  the classical Sobolev embedding  theorem:
\begin{equation}\label{qian}\begin{split}
&L^2([0,T];H^1)\cap W^{1,2}([0,T];H^{-1})\hookrightarrow C([0,T];L^2),
\end{split}
\end{equation}
 yields that for any $0<\tau<T$,
\begin{equation}\label{utlianxu}
tu_t\in C([0,T];H^1) \quad \text{and}\quad u_t\in C([\tau,T]\times \mathbb R).
\end{equation}

Finally, it remains to  show that
 $$ u_{xx} \in C([\tau,T]\times \mathbb{R})$$ 
 for any $0<\tau<T$. Actually, it follows from the fact $\rho>0$, \eqref{jichulianxu} and \eqref{utlianxu} that
$$u_{xx}=\frac{1}{\alpha}\rho^{1-\delta}\big(u_t+u u_x+\phi_x-\alpha\psi u_x\big)\in C([\tau,T]\times \mathbb R).$$

Until now,  we have  completed  the proof of Theorem \ref{th2}.

\subsection{Proof of Theorem \ref{th1}}
The proof of Theorem \ref{th1}  can be achieved via the  similar argument used  in Section \ref{se46}.

\section{Non-existence of global solutions with $L^\infty$ decay on $u$}

This section will be devoted to providing   the  proofs for  the non-existence theories of global regular solutions with $L^\infty$ decay on $u$ shown in Theorem \ref{th:2.20} and Corollary \ref{th:2.20-c}.

\subsection{Proof of Theorem \ref{th:2.20}}
Now we are ready to prove Theorem \ref{th:2.20}.  Let $T>0$ be any constant, and    $(\rho,u)\in D(T)$. It  follows from the definitions of  $m(t)$, $\mathbb{P}(t)$ and $E_k(t)$ that
$$
 |\mathbb{P}(t)|\leq \int \rho(t,x)|u(t,x)|\text{d}x\leq  \sqrt{2m(t)E_k(t)},
$$
which, together with the definition of the solution class $D(T)$, implies that
$$
0<\frac{|\mathbb{P}(0)|^2}{2m(0)}\leq E_k(t)\leq \frac{1}{2} m(0)|u(t,\cdot)|^2_\infty \quad \text{for} \quad t\in [0,T].
$$
Then one obtains that there exists a positive constant 
$$C_u=\frac{|\mathbb{P}(0)|}{m(0)}$$
such that
$$
|u(t,\cdot)|_\infty\geq C_u  \quad \text{for} \quad t\in [0,T].
$$
Thus one obtains     the desired conclusion as shown in Theorem \ref{th:2.20}.

\subsection{Proof of Corollary \ref{th:2.20-c}}

Let $ (\rho,u)(t,x)$ in $[0,T]\times \mathbb{R}$ be the global  regular solution obtained in Theorem \ref{th2} or \ref{th1}. Next we just need to show that $(\rho,u)\in D(T)$.

Actually,   the conservation of the total mass and the energy equality  have been shown in Lemma \ref{ms} and Appendix B, and  the conservation of the total  momentum can be verified as follows.
\begin{lem}
\label{lemmak} For any $T>0$, it holds that 
$$  \mathbb{P}(t)=\mathbb{P}(0) \quad  {\rm{for}} \quad t\in [0,T]. $$
\end{lem}
\begin{proof}
First, one can obtain that 
\begin{equation}\label{finite}\begin{split}
\mathbb{P}(t)=\int  \rho u\text{d}x \leq C|\rho|_1|u|_\infty<&\infty.
\end{split}
\end{equation}

Second, the momentum equation $\eqref{e1.1}_2$ implies  that
\begin{equation}\label{deng1}
\mathbb{P}_t=-\int (\rho u^2 )_x\text{d}x-\int P_x\text{d}x+\int (\mu(\rho)u_x)_x\text{d}x=0,
\end{equation}
where one has used the fact that
$$
\rho u^2(t,\cdot),\quad \rho^\gamma(t,\cdot) \quad \text{and} \quad \rho^\delta u_x(t,\cdot) \in W^{1,1}(\mathbb{R}).
$$
\end{proof}


According to Theorem \ref{th:2.20} and Lemma \ref{lemmak}, the proof of  Corollary \ref{th:2.20-c} is complete.

\appendix

\section{ Basic lemmas}

In this appendix, we list some useful lemmas which were used frequently in the previous sections.
The first one is the  well-known Gagliardo-Nirenberg inequality.
\begin{lem}[\cite{ln}]\label{al1}
Let  $1\le q, \ r\le \infty$, and $p$ satisfies $$\frac{1}{p}=\frac{j}{d}+\zeta\Big(\frac{1}{r}-\frac{m}{d}\Big)+(1-\zeta)\frac{1}{q},\quad \text{and}\quad \frac{j}{m}\le \zeta\le 1.$$ 
Then there exists a generic  constant $C>0$ depends only on $j,m,d,q,r$ and $\zeta$ such that for $f\in L^q\cap D^{m,r}(\mathbb R^d)$,
\begin{equation}\label{ae1}
    |D^j f|_p\leq C|D^m f|_r^\zeta |f|_q^{1-\zeta},
\end{equation}
with the exceptions that if $j=0, rm<d, q=\infty$ we assume that $f$ vanishes at infinity or $f\in L^{\tilde q}$ for some finite $\tilde q>0$, while if $1<r<\infty$ and $m-j-d/r$ is a non-negative integer we take $j/m\le \zeta<1$.

In particular, letting $d=1$, $j=0$ and  $p=\infty$ in \eqref{ae1}, one has 
\begin{equation*}
    |f|_\infty\leq C|f|_2^{\frac12}|f_x|_2^{\frac12}.
\end{equation*}
\end{lem}

The second  one is  Fatou's lemma.
\begin{lem}\label{Fatou}
Given a measure space $(V,\mathcal{F},\nu)$ and a set $X\in \mathcal{F}$, let  $\{f_n\}$ be a sequence of $(\mathcal{F} , \mathcal{B}_{\mathbb{R}_{\geq 0}} )$-measurable non-negative functions $f_n: X\rightarrow [0,\infty]$. Define the function $f: X\rightarrow [0,\infty]$ by setting
$$
f(x)= \liminf_{n\rightarrow \infty} f_n(x),
$$
for every $x\in X$. Then $f$ is $(\mathcal{F},  \mathcal{B}_{\mathbb{R}_{\geq 0}})$-measurable, and
$$
\int_X f(x) \text{\rm d}\nu \leq \liminf_{n\rightarrow \infty} \int_X f_n(x) \text{\rm d}\nu.
$$
\end{lem}

The following lemma is used to obtain the time-weighted estimates of the velocity $u$.
\begin{lem}\cite{bjr}\label{bjr}
If $f(t,x)\in L^2([0,T]; L^2)$, then there exists a sequence $s_k$ such that
$$
s_k\rightarrow 0 \quad \text{and}\quad s_k |f(s_k,\cdot)|^2_2\rightarrow 0 \quad \text{as} \quad k\rightarrow\infty.
$$
\end{lem}

The last one shows some compactness results obtained via the Aubin-Lions Lemma.
\begin{lem}[\cite{js}]
 Let $X_0\subset X\subset X_1$ be three Banach spaces.  Suppose that $X_0$ is compactly embedded in $X$ and $X$ is continuously embedded in $X_1$. Then the following statements hold
\begin{enumerate}
\item If $J$ is bounded in $L^p([0,T];X_0)$ for $1\leq p < \infty$, and $\frac{\partial J}{\partial t}$ is bounded in $L^1([0,T];X_1)$, then $J$ is relatively compact in $L^p([0,T];X)$.\\

\item If $J$ is bounded in $L^\infty([0,T];X_0)$  and $\frac{\partial J}{\partial t}$ is bounded in $L^p([0,T];X_1)$ for $p>1$, then $J$ is relatively compact in $C([0,T];X)$.
\end{enumerate}
\end{lem}

\section{  Proofs of  Lemmas \ref{l4.1}-\ref{l4.2} and the  equation \eqref{e-1.16}}

In this appendix, for the sake of completeness, we will give the  proof of  Lemmas \ref{l4.1}-\ref{l4.2} and the derivation of  the equation \eqref{e-1.16}. 

{\textbf{Proof of  Lemma \ref{l4.1}}}.\ 
It follows from  $\eqref{e1.1}_1$ that 
\begin{equation}\label{Peqn}
    \frac{P(\rho)_t}{\gamma-1} +\frac{(P(\rho) u)_x}{\gamma-1}+P(\rho) u_x=0.
\end{equation}
Then multiplying $\eqref{e1.1}_2$  by $u$ and adding the resulting equation to \eqref{Peqn}, and  integrating over $\mathbb{R}$, one arrives at
\begin{equation}\label{pin}
\frac{\text{d}}{\text{d}t} \int \left(\frac12 \rho u^2+\frac{P(\rho)}{\gamma-1}\right)\text{d}x + \alpha\int \rho^{\delta} u_x^2 \text{d}x=0,
\end{equation}
where one has used the fact that
$$
\rho u u^2(t,\cdot),\quad \rho^\gamma u(t,\cdot) \quad \text{and} \quad \rho^\delta u_x u(t,\cdot) \in W^{1,1}(\mathbb{R}).
$$

Finally, integrating \eqref{pin} over $[0,t]$, one can get the desired estimate. 

The proof of Lemma \ref{l4.1} is complete.

\par\bigskip

{\textbf{Proof of Lemma \ref{l4.2}}}.\ 
According to $\eqref{e1.1}_1$ and  $\eqref{e1.1}_2$, one has
\begin{align}
   &\varphi(\rho)_{tx}+(\varphi(\rho)_x u)_x+(\rho\varphi'(\rho) u_x)_x=0,\label{e5.2}\\
   & \rho(u_t+uu_x)+P(\rho)_x=(\mu(\rho) u_{x})_x.\label{e5.1}
\end{align}
Multiplying \eqref{e5.2} by $\rho$, one has 
\begin{equation}\label{es}
\rho \varphi(\rho)_{tx}+\rho u\varphi(\rho)_{xx}+(\rho^2\varphi'(\rho) u_x)_x=0.
\end{equation}
Then adding \eqref{es}  to  \eqref{e5.1}, by the definition of $v$,  one can obtain 
\begin{equation}\label{e5.3}
  \rho(v_t+u v_x)+P(\rho)_x-\left((\mu(\rho)-\rho^2\varphi'(\rho))u_{x}\right)_x=0,
\end{equation}
which, along with 
$$\varphi'(\rho)=\mu(\rho)/\rho^2$$
in \eqref{e5.3}, yields that 
 \begin{equation}\label{ess}
  \rho(v_t+u v_x)+P(\rho)_x=0.
\end{equation}
Multiplying \eqref{ess} by $v$ and integrating over $\mathbb{R}$, one has 
\begin{equation}\label{eso}
   \frac{\text{d}}{\text{d}t}\int\left(\frac12 \rho v^2+\frac{A}{\gamma-1}\rho^\gamma\right) {\rm{d}}x+\int P(\rho)_x \varphi(\rho)_x {\rm{d}}x=0.
\end{equation}
Integrating \eqref{eso} over $[0,T]$,  one can obtain the  desired conclusion.

The proof of Lemma \ref{l4.2} is complete. 
\par\bigskip

{\textbf{Proof of  \eqref{e-1.16}}}.\ 
It follows from \eqref{ess} that 
\begin{equation}
\begin{split}
v_t+u v_x=&-\frac{P(\rho)_x}{\rho}=-A\gamma\rho^{\gamma-2}\rho_x\\
=&-A\gamma\rho^{\gamma-\delta}\rho^{\delta-2}\rho_x=-\frac{A\gamma}{\alpha}\rho^{\gamma-\delta}(v-u),
\end{split}
\end{equation}
which implies  \eqref{e-1.16}.

%

\section{  Local-in-time well-posedness for multi-dimensional case}

In this appendix, for the sake of completeness, we will recall the   local-in-time well-posedness theory (obtained in \cite{sz3,zz2})  of regular solutions with far field vacuum  for the corresponding 3-D and 2-D Cauchy problems of the degenerate viscous flows.  
\subsection{$0<\delta<1$}
For the case $0<\delta<1$, we  first give the definition of regular solutions with far field vacuum  to the 3-D Cauchy problem \eqref{eq:1.1}-\eqref{10000} with the following initial data and far field behavior:
 \begin{equation}\label{ccc}
\begin{aligned}
(\rho(0,x),u(0,x))=(\rho_0(x),u_0(x)) \quad  & \text{for}\quad  x\in\mathbb R^3,\\[4pt]
\displaystyle
\left(\rho(t,x),u(t,x)\right)\to \left(0,0\right)\quad  \text{as}\quad  \left|x\right|\to \infty\quad & \text{for}\quad   t\ge 0.
\end{aligned}
\end{equation}

\begin{mydef}[\cite{zz2}]
Assume $0<\delta<1$ and  $T> 0$.  The pair $(\rho,u)$ is called a regular solution to the Cauchy problem  \eqref{eq:1.1}-\eqref{10000} with  \eqref{ccc}  in $[0,T]\times \mathbb R^3$, if $(\rho, u)$ satisfies this problem in the sense of distributions and:
\begin{equation*}
\begin{aligned}
&(1)\ \rho>0,\quad \rho^{\gamma-1}\in C([0,T];H^3),\quad  \nabla \rho^{\delta-1}\in L^\infty([0,T];L^\infty \cap D^2);\notag\\[1pt]
&(2)\ u\in C([0,T];H^3)\cap L^2([0,T]; H^4), \quad u_t\in C([0,T];H^1)\cap L^2([0,T]; D^2),\notag\\[1pt]
& \ \ \quad \rho^{\frac{\delta-1}{2}}\nabla u\in C([0,T]; L^2),\quad \rho^{\frac{\delta-1}{2}}\nabla u_{t}\in L^\infty([0,T];L^2),\notag\\[1pt]
& \ \ \quad \rho^{\delta-1} \nabla u \in L^\infty([0,T];D^1_*),\quad \rho^{\delta-1}\nabla^2 u\in C([0,T];H^1)\cap L^2([0,T];D^2).
\end{aligned}
\end{equation*}
\end{mydef}

 The corresponding local-in-time well-posedness can be stated as follows.

\begin{lem}[\cite{zz2}]\label{apd1}
Let  parameters   $\gamma$, $\delta$, $\alpha$ and  $\beta$ satisfy
\begin{equation*}
\gamma>1,\quad 0<\delta<1, \quad \alpha>0, \quad 2\alpha+3\beta\geq 0.
\end{equation*}
  Assume the initial data $( \rho_0, u_0)$ satisfy
\begin{equation*}
\rho_0>0,\quad (\rho^{\gamma-1}_0, u_0)\in H^3, \quad \nabla \rho^{\delta-1}_0\in D^1_*\cap D^2,   \quad \nabla \rho^{\frac{\delta-1}{2}}_0\in L^4,
\end{equation*}
and the initial  compatibility conditions
\begin{equation*}
  \nabla u_0=\rho^{\frac{1-\delta}{2}}_0 g_1,\quad \   \ \  Lu_0= \rho^{1-\delta}_0g_2,\quad \  \ \
 \nabla \left(\rho^{\delta-1}_0Lu_0\right)=\rho^{\frac{1-\delta}{2}}_0g_3,
\end{equation*}
for some $g_i(i=1,2,3)\in L^2$, then there exist a  time $T_*>0$ and a unique regular solution $(\rho, u)$  in $ [0,T_*]\times \mathbb{R}^3$ to the Cauchy problem \eqref{eq:1.1}-\eqref{10000} with \eqref{ccc} satisfying
\begin{equation*}
\begin{split}
& \rho^{\frac{\delta-1}{2}}\nabla u_t \in  L^\infty([0,T_*];L^2),\quad  \rho^{\delta-1}\nabla^2 u_t \in L^2([0,T_*]; L^2),\\
& t^{\frac{1}{2}}u\in L^\infty([0,T_*];D^4),\quad  t^{\frac{1}{2}}u_t\in L^\infty([0,T_*];D^2)\cap L^2([0,T_*] ; D^3),\\
& u_{tt}\in L^2([0,T_*];L^2),\quad   t^{\frac{1}{2}}u_{tt}\in L^\infty([0,T_*];L^2)\cap L^2([0,T_*];D^1_*),\\
& \rho^{1-\delta}\in L^\infty([0,T_*];  L^\infty\cap D^{1,6} \cap D^{2,3} \cap D^3), \\
&  \nabla \rho^{\delta-1}\in C([0,T_*]; D^1_*\cap D^2), \quad  \nabla \ln \rho \in L^\infty([0,T_*];L^\infty\cap L^6\cap  D^{1,3} \cap D^2).
\end{split}
\end{equation*}
\end{lem}

\subsection{$\delta=1$}
For the case $\delta=1$, we  first give the definition of regular solutions with far field vacuum  to the 2-D Cauchy problem of \eqref{eq:1.1}-\eqref{10000} with the following initial data and far field behavior:
\begin{equation}\label{cccc}
\begin{aligned}
(\rho(0,x),u(0,x))=(\rho_0(x),u_0(x)) \quad  & \text{for}\quad  x\in\mathbb R^2,\\[4pt]
\displaystyle
\left(\rho(t,x),u(t,x)\right)\to \left(0,0\right)\quad  \text{as}\quad  \left|x\right|\to \infty\quad & \text{for}\quad   t\ge 0.
\end{aligned}
\end{equation} 

\begin{mydef}[\cite{sz3}]
Assume $\delta=1$  and $T> 0$. The pair $(\rho,u)$ is called a regular solution to the Cauchy problem  \eqref{eq:1.1}-\eqref{10000} with \eqref{cccc} in $[0,T]\times \mathbb R^2$, if $(\rho, u)$ satisfies this problem in the sense of distributions and:
\begin{equation*}
\begin{aligned}
&(1)\ 0< \rho\in C^1([0,T]\times \mathbb R^2),\  \rho^{\frac{\gamma-1}{2}}\in C([0,T];H^3),\   \nabla \ln\rho\in C([0,T];L^6\cap D^1\cap D^2); \notag\\[1pt]
&(2)\ u\in C([0,T];H^3)\cap L^2([0,T]; H^4),\quad u_t\in C([0,T];H^1)\cap L^2([0,T]; D^2);\notag\\[1pt]
&(3)\  \lim_{|x|\to \infty}\left(u_t+ u\cdot\nabla u+Lu \right)=\lim_{|x|\to\infty} \left(\nabla\ln\rho\cdot Q(u)\right) \quad {\rm{for}}\quad  t>0.
\end{aligned}
\end{equation*}
\end{mydef}

The corresponding local-in-time well-posedness  can be stated  as follows.

\begin{lem}[\cite{sz3}]\label{apd2}
Let  parameters  $\gamma$, $\delta$, $\alpha$ and  $\beta$ satisfy
\begin{equation*}
\gamma>1,\quad \delta=1, \quad \alpha>0, \quad \alpha+\beta\geq 0.
\end{equation*}
Assume the initial data $( \rho_0, u_0)$ satisfy 
\begin{equation*}
\rho_0> 0,\quad (\rho^{\frac{\gamma-1}{2}}_0, u_0)\in H^3, \quad \nabla \ln\rho_0\in L^6 \cap  D^1\cap D^2,
\end{equation*}
 then there exist a time $T_*>0$ and a unique regular solution $(\rho, u)$   in $ [0,T_*]\times \mathbb{R}^2$  to the Cauchy problem \eqref{eq:1.1}-\eqref{10000} with \eqref{cccc} satisfying
\begin{equation*}
\begin{split}
& (\rho^{\frac{\gamma-1}{2}})_t \in C([0,T_*];H^2),\quad (\nabla \ln\rho)_t \in C([0,T_*]; H^1),\\
& u_{tt}\in L^2([0,T_*];L^2),\quad  t^{\frac{1}{2}}u\in L^\infty([0,T_*];D^4),\\
&t^{\frac{1}{2}}u_t\in L^\infty([0,T_*];D^2)\cap L^2([0,T_*] ; D^3),\ \  t^{\frac{1}{2}}u_{tt}\in L^\infty([0,T_*];L^2)\cap L^2([0,T_*];D^1).
\end{split}
\end{equation*}
\end{lem}

\bigskip
\noindent{\bf Acknowledgement:}
The research  was  supported in part  by  National Natural Science Foundation of China under Grants  12101395 and 12161141004. The research of S. Zhu was also supported partially by   The Royal Society  under Grants: Newton International Fellowships NF170015,  Newton International Fellowships Alumni AL/201021 and AL/211005.

\bigskip
\noindent{\bf Conflict of Interest:} The authors declare  that they have no conflict of
interest. 
The authors also  declare that this manuscript has not been previously  published, and will not be submitted elsewhere before your decision.

\end{document}